\documentclass[12pt]{article}
\usepackage[]{amsmath,amssymb}
\usepackage{amscd}
\usepackage{latexsym}
\usepackage{cite}
\usepackage{amsthm}

\newtheorem{definition}{Definition}[section]
\newtheorem{theorem}[definition]{Theorem}
\newtheorem{lemma}[definition]{Lemma}

\newtheorem{remark}[definition]{Remark}

\newtheorem{conjecture}[definition]{Conjecture}

\newtheorem{proposition}[definition]{Proposition}

\begin{document}
\title{\bf 
A conjecture concerning\\ the
 $q$-Onsager algebra}
\author{
Paul Terwilliger 
}
\date{}
%\footnote{This author gratefully acknowledges 
%support from the FY2007 JSPS Invitation Fellowship Program
%for Reseach in Japan (Long-Term), grant L-07512.}
%}
%\date{}
%to get date printout, comment out above line

\maketitle
\begin{abstract}
The $q$-Onsager algebra $\mathcal O_q$ is defined by two generators $W_0, W_1$ and two relations called the $q$-Dolan/Grady relations.
Recently Baseilhac and Kolb obtained a PBW basis for  $\mathcal O_q$ with elements denoted
\begin{align*}
\lbrace B_{n \delta+ \alpha_0} \rbrace_{n=0}^\infty,
\qquad \quad 
\lbrace B_{n \delta+ \alpha_1} \rbrace_{n=0}^\infty,
\qquad \quad 
\lbrace B_{n \delta} \rbrace_{n=1}^\infty.
\end{align*}
\noindent  In their recent study of a current algebra $\mathcal A_q$, Baseilhac and Belliard  conjecture that there exist elements
\begin{align*}
 \lbrace W_{-k}\rbrace_{k=0}^\infty, \qquad \lbrace  W_{k+1}\rbrace_{k=0}^\infty, \qquad  
 \lbrace  G_{k+1} \rbrace_{k=0}^\infty,
\qquad
\lbrace   {\tilde G}_{k+1} \rbrace_{k=0}^\infty
\end{align*}
in $\mathcal O_q$ that  satisfy the defining relations for $\mathcal A_q$. In order to establish this conjecture, it is desirable to know how the elements 
 on the second displayed line above are related to  the elements on the first displayed line above.
 In the present paper, we conjecture the precise relationship and give some supporting evidence.  This evidence consists of some computer checks on SageMath due to Travis Scrimshaw,
 a proof of the analog  conjecture for the Onsager algebra $\mathcal O$,
and a proof of our conjecture for a homomorphic image of $\mathcal O_q$ called the universal Askey-Wilson algebra.
\medskip

\bigskip


 \end{abstract}

\section{Introduction}
We will be discussing the $q$-Onsager algebra $\mathcal O_q$ \cite{bas1,qSerre}. This infinite-dimensional associative algebra is defined by two generators $W_0$, $W_1$ and two relations called the
$q$-Dolan/Grady relations; see Definition \ref{def:U}
below. One can view $\mathcal O_q$ as a $q$-analog of the universal enveloping algebra of the Onsager Lie algebra $\mathcal O$ \cite{Davfirst, Dav, Dolgra, Onsager, perk, perk2, roan}.
\medskip

\noindent
The algebra $\mathcal O_q$ originated in algebraic combinatorics \cite{qSerre}.
There is a family of algebras called tridiagonal algebras
\cite[Definition~3.9]{qSerre} that arise in the study of
association schemes
\cite[Lemma~5.4]{tersub3} and tridiagonal pairs
\cite[Theorem~10.1]{TD00},
\cite[Theorem~3.10]{qSerre}.
The algebra $\mathcal O_q$ is the ``most general'' example
of a tridiagonal algebra
\cite[Section~1.2]{ItoTerAug}. A finite-dimensional irreducible $\mathcal O_q$-module is essentially the same thing as a tridiagonal pair of $q$-Racah type \cite[Theorem~3.10]{qSerre}. 
These tridiagonal pairs are classified up to isomorphism in \cite[Theorem~3.3]{ItoTer}. To our knowledge the $q$-Dolan/Grady relations first appeared in \cite[Lemma~5.4]{tersub3}.
\medskip

\noindent  The algebra $\mathcal O_q$ has applications outside combinatorics.
 For instance, $\mathcal O_q$ is used 
 to study boundary
integrable systems 
\cite{
bas2,
bas1,
bas8,
basXXZ,
basBel,
BK05,
bas4,
basKoi,
basnc}.
The algebra $\mathcal O_q$  can be realized as a left or right coideal subalgebra of the quantized enveloping algebra
$U_q(\widehat{\mathfrak{sl}}_2)$; see \cite{bas8, basXXZ, kolb}. The algebra $\mathcal O_q$ is the simplest example of a quantum symmetric pair coideal subalgebra 
of affine type \cite[Example~7.6]{kolb}. 
A Drinfeld type presentation of $\mathcal O_q$ is obtained in \cite{LuWang}, and this  is used in \cite{LRW} to realize $\mathcal O_q$ as an $\iota$Hall algebra of the projective line.
There is an injective algebra homomorphism from $\mathcal O_q$ into
the algebra $\square_q$ 
\cite[Proposition~5.6]{pospart},
and a noninjective algebra  homomorphism from $\mathcal O_q$ into the universal Askey-Wilson algebra $\Delta_q$
\cite[Sections~9,10]{uaw}.
In \cite[Section~4]{basXXZ} some infinite-dimensional $\mathcal O_q$-modules are constructed using $q$-vertex operators. 
In \cite{ItoTerAug} the augmented $q$-Onsager algebra  is introduced; this algebra is obtained from $\mathcal O_q$ by adding an extra generator.
The augmented $q$-Onsager algebra is used in \cite{BT17} to derive a $Q$-operator.
In \cite{bas8} a higher rank generalization of $\mathcal O_q$ is introduced, and applied to affine Toda theories with boundaries.
\medskip

\noindent 
In \cite[Theorem~4.5]{BK},
Baseilhac and Kolb obtain a Poincar\'e-Birkhoff-Witt (or PBW)
basis for $\mathcal O_q$. They obtain this PBW basis by using a method of Damiani 
\cite{damiani}
along with
two automorphisms of
$\mathcal O_q$ that are roughly analogous to the
Lusztig automorphisms of
$U_q(\widehat{\mathfrak{sl}}_2)$.
The PBW basis elements are denoted
\begin{align}
\lbrace B_{n \delta+ \alpha_0} \rbrace_{n=0}^\infty,
\qquad \quad 
\lbrace B_{n \delta+ \alpha_1} \rbrace_{n=0}^\infty,
\qquad \quad 
\lbrace B_{n \delta} \rbrace_{n=1}^\infty.
\label{eq:UpbwIntro}
\end{align}
\noindent
%%%%%%%In \cite{BK05} Baseilhac and Koizumi introduce a current algebra $\mathcal A_q$ for $\mathcal O_q$, in order to solve boundary integrable systems with hidden symmetries.
In mathematical physics, $\mathcal O_q$ comes up naturally in the context of a reflection algebra \cite{bas2,bas1}. %% [2,3]. 
Using a framework of Sklyanin \cite{sklyanin}, in \cite{BK05, basnc} %[8,12]
 a current algebra $\mathcal A_q$ for $\mathcal O_q$ is introduced.
In \cite[Definition~3.1]{basnc} Baseilhac and Shigechi give a presentation of $\mathcal A_q$ by generators and relations. The generators are denoted
\begin{align*}
 \lbrace \mathcal W_{-k}\rbrace_{k=0}^\infty, \qquad \lbrace \mathcal W_{k+1}\rbrace_{k=0}^\infty, \qquad  
 \lbrace \mathcal G_{k+1} \rbrace_{k=0}^\infty,
\qquad
\lbrace  \mathcal {\tilde G}_{k+1} \rbrace_{k=0}^\infty
\end{align*}
and the relations are given in 
\eqref{eq:3p1}--\eqref{eq:3p11} below.
\medskip

\noindent We now summarize some recent results about $\mathcal A_q$.
 In \cite[Section~3]{basXXZ} a reflection algebra is used to obtain
 a generating function for quantities in a commutative subalgebra of $\mathcal A_q$.
%% Using the reflection algebra, a generating function for quantities in a commutative subalgebra of $\mathcal A_q$ is obtained
%%% For \calAq [(3.12) with (3.9) and first line of (3.11),4]. 
  In \cite{BK05, bas4} 
 some finite-dimensional tensor product representations of $\mathcal A_q$ are constructed, and used %[8,9],
 to create quantum integrable spin chains. 
 The algebra $\mathcal A_q$ is used to study the open XXZ spin chain with generic nondiagonal boundary conditions \cite{bas4, basKoi} % [9,10]
  and also its thermodynamic limit \cite{basXXZ,  BKoj14, BKoj14B}. %[5,BK1,BK2] 
 In \cite{BKoj14, BKoj14B} the study of $\mathcal A_q$ is combined with the $q$-vertex operator approach of the Kyoto school,  to derive correlation functions and form factors.
For the open XXZ spin chain in the thermodynamic limit, the algebra $\mathcal A_q$ is used in \cite{BB16} to classify the non-abelian symmetries for any type of boundary condition. %%%BB17] 
%In \cite{ItoTerAug} the augmented $q$-Onsager algebra  is introduced; this algebra is obtained from $\mathcal O_q$ by adding an extra generator.
%The augmented $q$-Onsager algebra is used in \cite{BT17} to derive a $Q$-operator.
%In \cite{bas8} a higher rank generalization of $\mathcal O_q$ is introduced, and applied to affine Toda theories with boundaries.
In \cite{BBC}, a limit $q\mapsto 1$ is taken in $\mathcal O_q$ to obtain a presentation of the Onsager algebra $\mathcal O$ in terms of a non-standard Yang-Baxter algebra.
%%%In [BBC17], a presentation of the Onsager algebra $\mathcal O$ (the specialization at q = 1 of $\mathcal O_q$) is given in terms of non-standard Yang-Baxter algebras.
 In \cite{BC17}, a similar limiting process is applied to $\mathcal A_q$, to obtain a Lie algebra $\mathcal A$ that turns out to be isomorphic to $\mathcal O$.
 An explicit isomorphism between $\mathcal O$ and $\mathcal A$ is established, and explicit relations between the generators of $\mathcal O$ and $\mathcal A$ are given.
\medskip

\noindent
 The algebras $\mathcal A_q$ and  $\mathcal O_q$ are both  $q$-analogs of the universal enveloping algebra of $\mathcal O$, %%% \cite[Definition~4.1, Theorem 2]{BC17},
 so it is natural to ask how $\mathcal A_q$ is related to $\mathcal O_q$. Baseilhac and Belliard investigate this issue in \cite{basBel}; their results are summarized as follows.
In \cite[line (3.7)]{basBel} they show that  $\mathcal W_0$, $\mathcal W_1$ satisfy the $q$-Dolan/Grady relations. 
In \cite[Section~3]{basBel} they show that $\mathcal A_q$ is generated by $\mathcal W_0$, $\mathcal W_1$ together with the
central elements $\lbrace \Delta_n \rbrace_{n=1}^\infty$ defined in
 \cite[Lemma~2.1]{basBel}.
In \cite[Section~3]{basBel} they consider the quotient algebra of $\mathcal A_q$ obtained
by sending $\Delta_n$ to a scalar for all $n\geq 1$. The construction yields an algebra homomorphism
 $\Psi$ from $\mathcal O_q$ onto this quotient.
In \cite[Conjecture~2]{basBel}
Baseilhac and Belliard conjecture that $\Psi$ is an isomorphism.  If the conjecture is true then there exists an algebra homomorphism  $\mathcal A_q\to \mathcal O_q$ that sends
$\mathcal W_0 \mapsto W_0$ and $\mathcal W_1 \mapsto W_1$. In this case there exist elements
\begin{align}
\label{eq:4gensIntro}
 \lbrace  W_{-k}\rbrace_{k=0}^\infty, \qquad \lbrace W_{k+1}\rbrace_{k=0}^\infty, \qquad  
 \lbrace G_{k+1} \rbrace_{k=0}^\infty,
\qquad
\lbrace  {\tilde G}_{k+1} \rbrace_{k=0}^\infty
\end{align}
 in $\mathcal O_q$ that satisfy the relations \eqref{eq:3p1}--\eqref{eq:3p11}.
In order to make progress on the above conjecture, it is desirable to know how the elements \eqref{eq:4gensIntro} are related to the elements in
\eqref{eq:UpbwIntro}. In the present paper, we conjecture the precise relationship and give some supporting evidence. 
Our conjecture statement is Conjecture  \ref{conj:M}. Our supporting evidence consists of some computer checks on SageMath (see \cite{sage}) due to Travis Scrimshaw, 
 a proof of the analog  conjecture for the Onsager algebra $\mathcal O$,
and a proof of the conjecture at the level of the algebra $\Delta_q$ mentioned above.
\medskip

\medskip

\noindent The paper is organized as follows. Section 2 contains some preliminaries. In Section 3 we recall the algebra $\mathcal O_q$,
and describe the PBW basis due to Baseilhac and Kolb.
In Sections 4, 5 we develop some results about generating functions that will be used in Conjecture \ref{conj:M}.
In Section 6 we state Conjecture \ref{conj:M} and explain its meaning.
In Section 7 we present our evidence supporting Conjecture \ref{conj:M}. In Section 8 we give some comments. In Appendices A, B we display
in detail some equations from the main body of the paper.
\medskip

\noindent 

\section{Preliminaries}
Throughout  the paper, the following notational conventions are in effect.
Recall the natural numbers 
$\mathbb N =\lbrace 0,1,2,\ldots \rbrace$ and integers $\mathbb Z = \lbrace 0, \pm 1, \pm 2, \ldots \rbrace$.
 Let $\mathbb F$ denote a field. Every vector space 
 mentioned is over $\mathbb F$. Every algebra mentioned is associative, over $\mathbb F$, and 
has a multiplicative identity. 
 \begin{definition}\rm 
(See \cite[p.~299]{damiani}.)
Let $ \mathcal A$ denote an algebra. A {\it Poincar\'e-Birkhoff-Witt} (or {\it PBW}) basis for $\mathcal A$
consists of a subset $\Omega \subseteq \mathcal A$ and a linear order $<$ on $\Omega$
such that the following is a basis for the vector space $\mathcal A$:
\begin{align*}
a_1 a_2 \cdots a_n \qquad n \in \mathbb N, \qquad a_1, a_2, \ldots, a_n \in \Omega, \qquad
a_1 \leq a_2 \leq \cdots \leq a_n.
\end{align*}
We interpret the empty product as the multiplicative identity in $\mathcal A$.
\end{definition}
\noindent We will be discussing generating functions. Let $\mathcal A$ denote an algebra and let $t$ denote an indeterminate.
For a sequence $\lbrace a_n \rbrace_{n \in \mathbb N}$ of elements in $\mathcal A$, the corresponding generating function is
%$a(t)= \sum_{n \in \mathbb N}a_n t^n$.
\begin{align*}
 a(t) = \sum_{n \in \mathbb N} a_n t^n.
 \end{align*} 
 The above sum is formal; issues of convergence are not considered.
 We call $a(t)$ the {\it generating function over $\mathcal A$ with coefficients $\lbrace a_n \rbrace_{n \in \mathbb N}$}.
  For generating functions
 $a(t)=\sum_{n \in \mathbb N} a_n t^n$ and
 $b(t) = \sum_{n \in \mathbb N} b_n t^n$ over $\mathcal A$, their product $a(t)b(t)$ is the generating function $\sum_{n \in \mathbb N}c_n t^n$  such that 
 $c_n = \sum_{i=0}^n a_i b_{n-i}$ for $n\in \mathbb N$.
 The set of generating functions over $\mathcal A$ forms an algebra.
 Let $a(t) = \sum_{n \in \mathbb N} a_n t^n$ denote a generating function over $\mathcal A$. We say that $a(t)$ is
 {\it normalized} whenever $a_0=1$. If  $0 \not= a_0 \in \mathbb F$ then define
\begin{align}
a(t)^\vee = a^{-1}_0 a(t),
\label{eq:avee}
\end{align}
and note that $a(t)^\vee$ is normalized.
 \medskip
 
 \noindent Fix a nonzero $q \in \mathbb F$
that is not a root of unity.
Recall the notation
\begin{align*}
\lbrack n\rbrack_q = \frac{q^n-q^{-n}}{q-q^{-1}}
\qquad \qquad n \in \mathbb N.
\end{align*}

\section{The $q$-Onsager algebra $\mathcal O_q$}
In this section we recall the $q$-Onsager algebra $\mathcal O_q$.
\noindent For elements $X, Y$ in any algebra, define their
commutator and $q$-commutator by 
\begin{align*}
\lbrack X, Y \rbrack = XY-YX, \qquad \qquad
\lbrack X, Y \rbrack_q = q XY- q^{-1}YX.
\end{align*}
\noindent Note that 
\begin{align*}
\lbrack X, \lbrack X, \lbrack X, Y\rbrack_q \rbrack_{q^{-1}} \rbrack
= 
X^3Y-\lbrack 3\rbrack_q X^2YX+ 
\lbrack 3\rbrack_q XYX^2 -YX^3.
\end{align*}

\begin{definition} \label{def:U} \rm
(See \cite[Section~2]{bas1}, \cite[Definition~3.9]{qSerre}.)
%%\rm (See \cite[Corollary~3.2.6]{lusztig}.) 
Define the algebra $\mathcal O_q$ by generators $W_0$, $W_1$ and relations
\begin{align}
\label{eq:qOns1}
&\lbrack W_0, \lbrack W_0, \lbrack W_0, W_1\rbrack_q \rbrack_{q^{-1}} \rbrack =(q^2-q^{-2})^2\lbrack W_1, W_0 \rbrack,
\\
\label{eq:qOns2}
&\lbrack W_1, \lbrack W_1, \lbrack W_1, W_0\rbrack_q \rbrack_{q^{-1}}\rbrack = (q^2-q^{-2})^2\lbrack W_0, W_1 \rbrack.
\end{align}
We call $\mathcal O_q$ the {\it $q$-Onsager algebra}.
%The generators $W_0, W_1$ are called {\it standard}.
The relations \eqref{eq:qOns1}, \eqref{eq:qOns2}  are called the {\it $q$-Dolan/Grady relations}.
\end{definition}
\begin{remark}
\rm In \cite{BK} Baseilhac and Kolb define the $q$-Onsager algebra in a slightly more general way that involves two scalar parameters $c, q$. Our $\mathcal O_q$ is their
$q$-Onsager algebra with $c=q^{-1}(q-q^{-1})^2$.
\end{remark}
\begin{remark}\label{rem:qEqual1}\rm We clarify how to recover the Onsager algebra $\mathcal O$ from $\mathcal O_q$ by taking a limit $q\mapsto 1$. To keep things simple, assume that
$\mathbb F=\mathbb C$. In \eqref{eq:qOns1}, \eqref{eq:qOns2} make a change of variables $W_0 = \xi A_0$ and $W_1 = \xi A_1$ with $\xi = \sqrt{-1}(q-q^{-1})/2$. Simplify and set $q=1$ to obtain
\begin{align*}
\lbrack A_0, \lbrack A_0, \lbrack A_0, A_1\rbrack \rbrack \rbrack =16\lbrack A_0, A_1 \rbrack,
\qquad \quad
\lbrack A_1, \lbrack A_1, \lbrack A_1, A_0\rbrack \rbrack\rbrack = 16\lbrack A_1, A_0 \rbrack.
\end{align*}
\noindent These are the Dolan/Grady relations and the defining relations for $\mathcal O$ \cite[Section~2.1]{BC17}.
\end{remark}
%\noindent We will be discussing automorphisms and antiautomorphisms.
 %For an algebra $\mathcal A$, an {\it automorphism} of $\mathcal A$ is an algebra isomorphism $\mathcal A \to \mathcal A$.
%The {\it opposite algebra} $\mathcal A^{\rm opp}$ consists of the vector space $\mathcal A$ and the multiplication map
%$\mathcal A \times \mathcal A \to \mathcal A$, $(a,b)\mapsto ba$.
%An {\it antiautomorphism} of $\mathcal A$ is an algebra isomorphism $\mathcal A \to \mathcal A^{\rm opp}$.
%\begin{lemma} 
%\label{lem:AAut}
%There exists a unique automorphism of $U^+_q$ that swaps $W_0, W_1$. Moreover, there exists a unique
%antiautomorphism of $U^+_q$ that fixes each of $W_0, W_1$. These two maps commute.
%\end{lemma}
%\begin{definition} \label{def:sig}
%\rm Let $\xi$ denote the composition of the two maps from Lemma
%\ref{lem:AAut}.  The map $\xi$ is an antiautomorphism of $U^+_q$ that swaps $W_0, W_1$.
%\end{definition}

\noindent In \cite{BK}, Baseilhac and Kolb obtain a PBW basis for $\mathcal O_q$ that involves some elements
\begin{align}
\lbrace B_{n \delta+ \alpha_0} \rbrace_{n=0}^\infty,
\qquad \quad 
\lbrace B_{n \delta+ \alpha_1} \rbrace_{n=0}^\infty,
\qquad \quad 
\lbrace B_{n \delta} \rbrace_{n=1}^\infty.
\label{eq:Upbw}
\end{align}
These elements are recursively defined  as follows. Writing $B_\delta  = q^{-2}W_1 W_0 - W_0 W_1$ we have
\begin{align}
&B_{\alpha_0}=W_0,  \qquad \qquad 
B_{\delta+\alpha_0} = W_1 + 
\frac{q \lbrack B_{\delta}, W_0\rbrack}{(q-q^{-1})(q^2-q^{-2})},
\label{eq:line1}
\\
&
B_{n \delta+\alpha_0} = B_{(n-2)\delta+\alpha_0}
+ 
\frac{q \lbrack B_{\delta}, B_{(n-1)\delta+\alpha_0}\rbrack}{(q-q^{-1})(q^2-q^{-2})} \qquad \qquad n\geq 2
\label{eq:line2}
\end{align}
and 
\begin{align}
&B_{\alpha_1}=W_1,  \qquad \qquad 
B_{\delta+\alpha_1} = W_0 - 
\frac{q \lbrack B_{\delta}, W_1\rbrack}{(q-q^{-1})(q^2-q^{-2})},
\label{eq:line3}
\\
&
B_{n \delta+\alpha_1} = B_{(n-2)\delta+\alpha_1}
- 
\frac{q \lbrack B_{\delta}, B_{(n-1)\delta+\alpha_1}\rbrack}{(q-q^{-1})(q^2-q^{-2})} \qquad \qquad n\geq 2.
\label{eq:line4}
\end{align}
Moreover for $n\geq 2$,
\begin{equation}
\label{eq:Bdelta}
B_{n \delta} = 
q^{-2}  B_{(n-1)\delta+\alpha_1} W_0
- W_0 B_{(n-1)\delta+\alpha_1}  + 
(q^{-2}-1)\sum_{\ell=0}^{n-2} B_{\ell \delta+\alpha_1}
B_{(n-\ell-2) \delta+\alpha_1}.
\end{equation}
By \cite[Proposition~5.12]{BK} the elements $\lbrace B_{n\delta}\rbrace_{n=1}^\infty$ mutually commute.

\begin{lemma}
\label{prop:damiani} %% {\rm (See \cite[p.~308]{damiani}.)} 
{\rm (See \cite[Theorem~4.5]{BK}.)}
Assume that $q$ is transcendental over $\mathbb F$. Then
 a PBW basis for $\mathcal O_q$ is obtained by the elements {\rm \eqref{eq:Upbw}} in any linear
order.
%\begin{align*}
%B_{\alpha_0} <
%B_{\delta+ \alpha_0} < 
%B_{2 \delta + \alpha_0} < 
%\cdots < 
%B_\delta < 
%B_{2\delta} < 
%B_{3\delta} < 
%\cdots < 
%B_{2\delta+ \alpha_1} <
%B_{\delta+ \alpha_1} < 
%B_{\alpha_1}.
%\end{align*}
\end{lemma}

\begin{remark} \label{rem:qEqual1A} \rm
With reference to Remark \ref{rem:qEqual1}, we give the limiting values of the elements \eqref{eq:Upbw}. In  \eqref{eq:line1}--\eqref{eq:Bdelta} and the expression for  $B_{\delta}$ below \eqref{eq:Upbw}, make a change of variables
\begin{align*}
B_{n\delta+\alpha_0} = \xi A_{-n}, \qquad \quad B_{n\delta+\alpha_1} = \xi A_{n+1},
\qquad \quad B_{m\delta} = 4 \xi^2 B_m
\end{align*}
 for $n\geq 0 $ and $m\geq 1$. Simplify and set $q=1$ to obtain
 \begin{align*}
 \lbrack B_1, A_n \rbrack = 2A_{n+1} -2A_{n-1}, \qquad \qquad 
 \lbrack A_m, A_0\rbrack = 4B_m 
 \end{align*}
 \noindent for $n \in \mathbb Z$ and $m\geq 1$.
 \noindent The elements $\lbrace A_n \rbrace_{n \in \mathbb Z}$, $\lbrace B_n \rbrace_{n=1}^\infty$ form the basis for $\mathcal O$ given in
 \cite[Definition~2.1]{BC17}.
\end{remark}

\begin{definition} \label{def:Bgen}
\rm We define a generating function in the indeterminate $t$:
\begin{align}
B(t) = \sum_{n\in \mathbb N} B_{n\delta} t^n, \qquad \qquad B_{0\delta} = q^{-2}-1.
 \label{eq:zerodelta}
 \end{align}
\end{definition}
\noindent In Section 6 we will make a conjecture about $B(t)$. In Sections 4, 5 we motivate the conjecture with some comments about generating functions.

\section{Generating functions over a commutative algebra}
\noindent Throughout  this section the following notational conventions are in effect.
We fix a commutative algebra  $\mathcal A$. Every  generating function  mentioned 
is  over $\mathcal A$.
 \medskip
 
 \noindent
 The following results are readily checked.
 \begin{lemma} \label{lem:inv} A generating function $a(t)=\sum_{n \in \mathbb N} a_n t^n$  is invertible if and only if $a_0$ is invertible in $\mathcal A$. In this case
 $(a(t))^{-1}=\sum_{n \in \mathbb N}b_n t^n$ where
 $ b_0 = a^{-1}_0$ and for $n\geq 1$,
 %%%%%%% $0 = \sum_{k=0}^n a_k b_{n-k}$ for $n\geq 1$.
 \begin{align*}
 b_n = - a_0^{-1}\sum_{k=1}^n a_k b_{n-k}.
 \end{align*}
 \end{lemma}
 
 \begin{lemma} \label{lem:abbpre} For generating functions $a(t)= \sum_{n \in \mathbb N} a_n t^n$ and $b(t)=\sum_{n\in \mathbb N} b_n t^n$ the following are equivalent:
 \begin{enumerate}
 \item[\rm (i)] $a(t)= b(qt) b(q^{-1}t)$;
 \item[\rm (ii)] 
 $a_n = \sum_{i=0}^n b_i b_{n-i} q^{2i-n}$ for $n \in \mathbb N$.
 \end{enumerate}
 \end{lemma}
 
%\begin{definition} \label{def:normalize}
%\rm A generating function $a(t)=\sum_{n \in \mathbb N} a_n t^n$ is called {\it normalized} whenever $a_0=1$.
%\end{definition}

\begin{lemma} \label{lem:qsqrt}
For a normalized generating function $a(t)= \sum_{n\in \mathbb N} a_n t^n$, there exists a unique normalized generating function $b(t)=\sum_{n\in \mathbb N} b_n t^n$ such that
\begin{align*}
a(t) = b(qt)b(q^{-1}t).
\end{align*}
 Moreover for $n\geq 1$,
\begin{align*}
 b_n = \frac{ a_n - \sum_{i=1}^{n-1} b_i b_{n-i} q^{2i-n}}{q^n+q^{-n}}.
 \end{align*}
 \end{lemma}
 \begin{definition}\label{def:qsqrt}
 \rm Referring to Lemma \ref{lem:qsqrt}, we call $b(t)$ the {\it $q$-square root} of $a(t)$.
 \end{definition}

 \begin{lemma} \label{lem:abbp} For generating functions $a(t)= \sum_{n \in \mathbb N} a_n t^n$ and $b(t)=\sum_{n\in \mathbb N} b_n t^n$ the following are equivalent:
 \begin{enumerate}
 \item[\rm (i)] $ a(t)= b\bigl(\frac{q+q^{-1}}{t+t^{-1}} \bigr)$;
 \item[\rm (ii)] $a_0=b_0$ and for $n\geq 1$,
 \begin{align}
 \label{eq:abp}
 a_n = \sum_{\ell=0}^{\lfloor (n-1) /2 \rfloor} (-1)^\ell \binom{n-1-\ell}{\ell} \lbrack 2 \rbrack_q^{n-2 \ell} b_{n-2\ell}.
\end{align}
 \end{enumerate}
 \end{lemma}
 \begin{proof} 
Note that for $k \in \mathbb N$,
\begin{align}
\label{eq:power2}
(1-t)^{-k-1} = \sum_{\ell \in \mathbb N} \binom{k+\ell}{\ell} t^\ell.
\end{align}
\noindent We have
\begin{align*}
b\biggl( \frac{q+q^{-1}}{t+t^{-1}}\biggr)
&= \sum_{n \in \mathbb N} 
\biggl( \frac{q+q^{-1}}{t+t^{-1}}\biggr)^n b_n
= 
b_0+
\sum_{k \in \mathbb N} 
\biggl( \frac{q+q^{-1}}{t+t^{-1}}\biggr)^{k+1} b_{k+1}.
\end{align*}
\noindent We have
\begin{align*}
\frac{q+q^{-1}}{t+t^{-1}}
= \lbrack 2 \rbrack_q t(1+t^2)^{-1}.
\end{align*}
By this and \eqref{eq:power2} we find that for $k \in \mathbb N$,
\begin{align*}
\biggl( \frac{q+q^{-1}}{t+t^{-1}}\biggr)^{k+1}
= 
 \lbrack 2 \rbrack^{k+1}_q t^{k+1}
\sum_{\ell \in \mathbb N} (-1)^\ell \binom{k+\ell}{\ell}  t^{2 \ell}.
\end{align*}
By these comments
\begin{align*}
b\biggl( \frac{q+q^{-1}}{t+t^{-1}}\biggr) &= b_0 + \sum_{k, \ell\in \mathbb N}
 (-1)^\ell \binom{k+\ell}{\ell} \lbrack 2 \rbrack^{k+1}_q b_{k+1} t^{k+1+2 \ell} \\
 &= b_0 + \sum_{n=1}^\infty \sum_{\ell=0}^{\lfloor (n-1) /2 \rfloor} (-1)^\ell \binom{n-1-\ell}{\ell} \lbrack 2 \rbrack_q^{n-2 \ell} b_{n-2\ell} t^n.
\end{align*}
The result follows.
\end{proof}
 
 \begin{lemma} \label{lem:qsym} 
 For a generating function $a(t)=\sum_{n \in \mathbb N} a_n t^n$, there exists a unique generating function $b(t)=\sum_{n\in \mathbb N} b_n t^n$ such that
\begin{align}
\label{eq:qsym}
a(t) = b\biggl(\frac{q+q^{-1}}{t+ t^{-1}}\biggr).
\end{align}
Moreover $b_0= a_0 $ and for $n\geq 1$,
\begin{align*}
 b_n = \frac{ a_n - \sum_{\ell=1}^{\lfloor (n-1) /2 \rfloor} (-1)^\ell \binom{n-1-\ell}{\ell} \lbrack 2 \rbrack_q^{n-2 \ell} b_{n-2\ell}}{\lbrack 2 \rbrack_q^n}.
\end{align*}
\end{lemma}
\begin{proof} This is a routine consequence of Lemma  \ref{lem:abbp}.
\end{proof}

 \begin{definition}\label{def:qsym}
 \rm Referring to Lemma
 \ref{lem:qsym}, we call $b(t)$ the {\it $q$-symmetrization} of $a(t)$.
 \end{definition}
 \noindent We now combine the above constructions.
 \begin{proposition} \label{prop:combine}
 Let $a(t) =\sum_{n\in \mathbb N} a_n t^n$ denote a normalized generating function. Then for a generating function $b(t)=\sum_{n\in \mathbb N} b_n t^n$ the following
 are equivalent:
 \begin{enumerate}
 \item[\rm (i)] $b(t)$ is the $q$-symmetrization of the $q$-square root of the inverse of $a(t)$;
\item[\rm (ii)] $b(t) $ is normalized and
 \begin{align}
  a(t) b\biggl( \frac{q+q^{-1}}{q t + q^{-1}t^{-1}} \biggr) b\biggl( \frac{q+q^{-1}}{ q^{-1} t + q t^{-1}} \biggr) = 1;
  \label{eq:abb}
  \end{align}
  \item[\rm (iii)] $b(t)$ is normalized and
  \begin{align}
  \label{eq:abab}
  a(qt) b\biggl( \frac{q+q^{-1}}{q^2 t + q^{-2}t^{-1}} \biggr) =a(q^{-1} t) b\biggl( \frac{q+q^{-1}}{q^{-2} t + q^{2}t^{-1}} \biggr);
  \end{align}
  \item[\rm (iv)] $b_0=1$ and for $n\geq 1$,
 \begin{align}
\label{eq:recBG1}
0 = \lbrack n \rbrack_q a_n  
+
\sum_{\stackrel{\scriptstyle j+k+2\ell+1=n,}{\scriptstyle j,k,\ell\geq 0}}
(-1)^\ell 
\binom{k+\ell}{\ell}
\lbrack 2n-j \rbrack_q
\lbrack 2 \rbrack^{k+1}_q
a_j b_{k+1}.
\end{align}
\end{enumerate}
  \end{proposition}
  \begin{proof}    ${{\rm (i)} \Rightarrow {\rm (ii)}}$ Let $a_1(t)$ denote the inverse of $a(t)$, and let $a_2(t)$ denote the $q$-square root of $a_1(t)$. By assumption $b(t)$ is the $q$-symmetrization of $a_2(t)$.
  The generating function $a(t)$ is normalized, so $a_1(t)$ is normalized by Lemma
  \ref{lem:inv}. Now $a_2(t)$ is normalized by Lemma
  \ref{lem:qsqrt} and Definition
  \ref{def:qsqrt}.
   Now $b(t)$ is normalized by Lemma  \ref{lem:qsym} and Definition \ref{def:qsym}.
  By construction 
  \begin{align*}
  a(t) a_1(t) = 1, \qquad \quad a_1(t) = a_2(qt) a_2 (q^{-1}t), \qquad \quad a_2(t) = b\biggl(\frac{q+q^{-1}}{t+t^{-1}}\biggr).
  \end{align*}
  Combining these equations we obtain \eqref{eq:abb}.
  \\
  \noindent 
   ${{\rm (ii)} \Rightarrow {\rm (iii)}}$ 
   In the equation
  \eqref{eq:abb}, replace $t$ by $qt$ and also by $q^{-1}t$. Compare the two resulting equations to obtain
  \eqref{eq:abab}.
  \\
  \noindent ${{\rm (iii)} \Rightarrow {\rm (iv)}}$ Write each side of \eqref{eq:abab} as a power series in $t$, and compare coefficients. 
  \\
  \noindent ${{\rm (iv)} \Rightarrow {\rm (i)}}$ By assumption, the generating function $b(t)$ is normalized and satisfies \eqref{eq:recBG1}. Let $b'(t)$ denote the 
  the $q$-symmetrization of the $q$-square root of the inverse of $a(t)$. From our earlier comments, the generating function $b'(t)$ is normalized and satisfies \eqref{eq:recBG1}.
  The equations
  \eqref{eq:recBG1} admit a unique solution, so $b(t)=b'(t)$.
 \end{proof}
 
 \begin{definition}
 \label{def:qexp} \rm Referring to Proposition \ref{prop:combine}, we call $b(t)$ the {\it $q$-expansion of $a(t)$} whenever the equivalent conditions (i)--(iv) are satisfied.
 \end{definition}

 \begin{lemma} 
\label{lem:com}
Let $a(t)=\sum_{n\in \mathbb N} a_n t^n$ denote a normalized generating function. Let $b(t)=\sum_{n\in \mathbb N} b_n t^n$ denote the $q$-expansion of $a(t)$.
Then for $n\geq 1$ the following hold:
\begin{enumerate}
\item[\rm (i)] $b_n$ is a polynomial  in $a_1, a_2,\ldots, a_n$ that has coefficients in $\mathbb F$ and total degree $n$, where we view $a_k$ as having degree $k$ for $1\leq k \leq n$.
In this polynomial the coefficient of $a_n$ is $-\lbrack n \rbrack_q \lbrack 2n \rbrack^{-1}_q \lbrack 2 \rbrack^{-n}_q$.
\item[\rm (ii)] $a_n$ is a polynomial in $b_1, b_2, \ldots, b_n$ that has coefficients in $\mathbb F$ and total degree $n$, where we view $b_k$ as having degree $k$ for $1 \leq k \leq n$.
In this polynomial the coefficient of $b_n$ is $-\lbrack n \rbrack^{-1}_q \lbrack 2n \rbrack_q \lbrack 2 \rbrack^{n}_q$.
\end{enumerate}
\end{lemma}
\begin{proof} (i) By \eqref{eq:recBG1} and induction on $n$.
\\
\noindent (ii) By (i) above and induction on $n$.
\end{proof}

\section{Generating functions over a noncommutative algebra}
Throughout this section the following notational conventions are in effect. We fix an algebra $\mathcal B$ that is not necessarily commutative.
Every generating function mentioned is over $\mathcal B$.

\begin{definition}\label{def:com}
\rm A generating function $a(t)=\sum_{n \in \mathbb N} a_n t^n $ is said to be {\it commutative} whenever
$\lbrace a_n \rbrace_{n \in \mathbb N}$ mutually commute.
\end{definition}

\begin{lemma} \label{lem:AB}
For a commutative generating function $a(t)=\sum_{n \in \mathbb N} a_n t^n $ there exists a commutative subalgebra $\mathcal A$
of $\mathcal B$ that contains $a_n $ for $n \in \mathbb N$.
\end{lemma}
\begin{proof} Take $\mathcal A$ to be the subalgebra of $\mathcal B$ generated by $\lbrace a_n \rbrace_{n \in \mathbb N}$.
\end{proof}

\noindent  Referring to Lemma \ref{lem:AB}, we may view $a(t)$ as a generating function over $\mathcal A$.

\begin{definition}\label{def:qexpC}
\rm Let $a(t)=\sum_{n \in \mathbb N} a_n t^n$ denote a generating function that is commutative and normalized.
 By the {\it  $q$-expansion of $a(t)$} we mean
the $q$-expansion of the generating function $a(t)$ over  $\mathcal A$, where $\mathcal A$ is from Lemma \ref{lem:AB}. By \eqref{eq:recBG1} and Lemma \ref{lem:com}, the $q$-expansion of $a(t)$ is  independent of
the choice of $\mathcal A$.
\end{definition}

\section{Some elements in $\mathcal O_q$}

\noindent In the previous two sections we discussed generating functions. We now return our attention to the $q$-Onsager algebra $\mathcal O_q$.
Recall from Section 1 that in \cite[Conjecture~2]{basBel}  Baseilhac and Belliard effectively conjecture that there exist elements
\begin{align}
\label{eq:4gens}
 \lbrace W_{-k}\rbrace_{k\in \mathbb N}, \qquad \lbrace W_{k+1}\rbrace_{k \in \mathbb N}, \qquad  
 \lbrace G_{k+1} \rbrace_{k \in \mathbb N},
\qquad
\lbrace  \tilde G_{k+1} \rbrace_{k \in \mathbb N}
\end{align}
in $\mathcal O_q$ that satisfy the following relations.
%(See 
%\cite{BK05}, \cite[Definition~3.1]{basnc}.
For $k,\ell \in \mathbb N$,
\begin{align}
&
 \lbrack  W_0,W_{k+1}\rbrack= 
\lbrack W_{-k}, W_{1}\rbrack=
({{\tilde G}}_{k+1} -  G_{k+1})/(q+q^{-1}),
\label{eq:3p1}
\\
&
\lbrack W_0, G_{k+1}\rbrack_q= 
\lbrack \tilde G_{k+1}, W_{0}\rbrack_q= 
\rho  W_{-k-1}-\rho 
 W_{k+1},
\label{eq:3p2}
\\
&
\lbrack  G_{k+1},  W_{1}\rbrack_q= 
\lbrack W_{1}, {{\tilde G}}_{k+1}\rbrack_q= 
\rho  W_{k+2}-\rho 
 W_{-k},
\label{eq:3p3}
\\
&
\lbrack W_{-k},  W_{-\ell}\rbrack=0,  \qquad 
\lbrack W_{k+1}, W_{\ell+1}\rbrack= 0,
\label{eq:3p4}
\\
&
\lbrack W_{-k}, W_{\ell+1}\rbrack+
\lbrack  W_{k+1}, W_{-\ell}\rbrack= 0,
\label{eq:3p5}
\\
&
\lbrack W_{-k}, G_{\ell+1}\rbrack+
\lbrack G_{k+1},  W_{-\ell}\rbrack= 0,
\label{eq:3p6}
\\
&
\lbrack W_{-k}, {\tilde G}_{\ell+1}\rbrack+
\lbrack {\tilde G}_{k+1}, W_{-\ell}\rbrack= 0,
\label{eq:3p7}
\\
&
\lbrack W_{k+1},  G_{\ell+1}\rbrack+
\lbrack  G_{k+1},  W_{\ell+1}\rbrack= 0,
\label{eq:3p8}
\\
&
\lbrack W_{k+1}, {\tilde G}_{\ell+1}\rbrack+
\lbrack {\tilde G}_{k+1}, W_{\ell+1}\rbrack= 0,
\label{eq:3p9}
\\
&
\lbrack G_{k+1}, G_{\ell+1}\rbrack=0,
\qquad 
\lbrack {\tilde G}_{k+1}, {\tilde G}_{\ell+1}\rbrack= 0,
\label{eq:3p10}
\\
&
\lbrack {\tilde G}_{k+1}, G_{\ell+1}\rbrack+
\lbrack G_{k+1}, {\tilde G}_{\ell+1}\rbrack= 0.
\label{eq:3p11}
\end{align}
In the above equations  $\rho = -(q^2-q^{-2})^2$. For notational convenience define
\begin{align}
\label{eq:GtG}
G_0 = -(q-q^{-1}) \lbrack 2 \rbrack^2_q, \qquad \qquad \tilde G_0 = -(q-q^{-1}) \lbrack 2 \rbrack^2_q.
\end{align}
\begin{remark}\label{rem:qEqual1C}
\rm Referring to Remark \ref{rem:qEqual1}, we give the limiting values of the elements \eqref{eq:4gens}.
In \eqref{eq:3p1}--\eqref{eq:3p11}, make a change of variables
\begin{align*}
W_{-k} =\xi W'_{-k}, \qquad
W_{k+1} =\xi W'_{k+1}, \qquad
G_{k+1} =\xi^2 G'_{k+1}, \qquad
{\tilde G}_{k+1} =\xi^2 {\tilde G}'_{k+1}
\end{align*}
for $k \in \mathbb N$. Simplify and set $q=1$.  Lines \eqref{eq:3p1}--\eqref{eq:3p3} become
\begin{align}
&
 \lbrack  W'_0,W'_{k+1}\rbrack= 
\lbrack W'_{-k}, W'_{1}\rbrack=
({\tilde G}'_{k+1} -  G'_{k+1})/2,
\label{eq:3p1alt}
\\
&
\lbrack W'_0, G'_{k+1}\rbrack= 
\lbrack \tilde G'_{k+1}, W'_{0}\rbrack= 
16 W'_{-k-1}-16
 W'_{k+1},
\label{eq:3p2alt}
\\
&
\lbrack  G'_{k+1},  W'_{1}\rbrack= 
\lbrack W'_{1}, {\tilde G}'_{k+1}\rbrack= 
16 W'_{k+2}-16
 W'_{-k}
\label{eq:3p3alt}
\end{align}
and \eqref{eq:3p4}--\eqref{eq:3p11} remain essentially unchanged. 
In \cite[Definition~4.1]{BC17} and \cite[Theorem~2]{BC17},  Baseilhac and Cramp\'e display a basis
$ \lbrace W'_{-k}\rbrace_{k\in \mathbb N}$, $\lbrace W'_{k+1}\rbrace_{k \in \mathbb N}$, $\lbrace {\tilde G}'_{k+1} \rbrace_{k \in \mathbb N}$ for $\mathcal O$ that satisfies
 \eqref{eq:3p4}--\eqref{eq:3p3alt}, where  $G'_{k+1}=-\tilde G'_{k+1}$ for $k \in \mathbb N$.
\end{remark}

%\begin{remark} \rm For some applications the limit $q\mapsto 1$ is important. To obtain this limit, view $\rho$ as an indeterminate and make a change of variables
% saying that the case q=1 can be considered through a change of definition of the generators such that $\rho\not= 0$at q = 1. - saying that at q = 1, the new defining relations of Aq coincide with the ones of A Definition 4.1, in [7].
 % If the author wants to be precise, the change of normalization is just: \cite[Definition~4.1]{BC17}
%\end{remark}
\noindent %So far it has not been established  that the elements (\ref{eq:4gens}) exist in $\mathcal O_q$. 
Returning to $\mathcal O_q$, it is desirable to know how the elements \eqref{eq:4gens} are related to  the elements \eqref{eq:Upbw}.  In this paper we conjecture the precise relationship. We will state the conjecture shortly.
Before stating the conjecture, we discuss what is involved. Let us simplify things by writing
the elements \eqref{eq:4gens} in terms of $W_0$, $W_1$, $\lbrace \tilde G_{k+1} \rbrace_{k\in \mathbb N}$. To do this, 
we use
\eqref{eq:3p2}, \eqref{eq:3p3} to recursively obtain
$W_{-k}$, $W_{k+1}$ for $k\geq 1$:
\begin{align*}
W_{-1} &= W_1 -\frac{\lbrack \tilde G_{1}, 
W_0\rbrack_q}{(q^2-q^{-2})^2},
\\
W_{3} &=  W_1 
-
\frac{\lbrack \tilde G_{1}, 
W_0\rbrack_q}{(q^2-q^{-2})^2}
-
\frac{\lbrack
 W_1,
\tilde G_{2} 
\rbrack_q}{(q^2-q^{-2})^2},
\\
W_{-3} &=  W_1 
-
\frac{\lbrack \tilde G_{1}, 
W_0\rbrack_q}{(q^2-q^{-2})^2}
-
\frac{\lbrack
W_1,
\tilde G_{2} 
\rbrack_q}{(q^2-q^{-2})^2}
-
\frac{\lbrack \tilde G_{3}, 
W_0\rbrack_q}{(q^2-q^{-2})^2},
\\
W_{5} &= W_1 
-
\frac{\lbrack \tilde G_{1}, 
W_0\rbrack_q}{(q^2-q^{-2})^2}
-
\frac{\lbrack
W_1,
\tilde G_{2} 
\rbrack_q}{(q^2-q^{-2})^2}
-
\frac{\lbrack \tilde G_{3}, 
W_0\rbrack_q}{(q^2-q^{-2})^2}
-
\frac{\lbrack
W_1,\tilde G_{4} 
\rbrack_q}{(q^2-q^{-2})^2},
\\
W_{-5} &= W_1 
-
\frac{\lbrack \tilde G_{1}, W_0\rbrack_q}{(q^2-q^{-2})^2}
-
\frac{\lbrack
W_1,
\tilde G_{2} 
\rbrack_q}{(q^2-q^{-2})^2}
-
\frac{\lbrack  {\tilde G}_{3}, 
 W_0\rbrack_q}{(q^2-q^{-2})^2}
-
\frac{\lbrack
W_1,
 \tilde G_{4} 
\rbrack_q}{(q^2-q^{-2})^2}
-
\frac{\lbrack  \tilde G_{5}, 
W_0\rbrack_q}{(q^2-q^{-2})^2},
\\
& \cdots
\end{align*}
\begin{align*}
W_2 &=  W_0 -
\frac{
\lbrack W_1, \tilde G_1\rbrack_q}{(q^2-q^{-2})^2
},
\\
W_{-2} &=  W_0 -
\frac{
\lbrack W_1, \tilde G_1\rbrack_q}{(q^2-q^{-2})^2}
-
\frac{
\lbrack \tilde G_2, W_0\rbrack_q}{(q^2-q^{-2})^2},
\\
W_{4} &=  W_0 -
\frac{
\lbrack W_1, \tilde G_1\rbrack_q}{(q^2-q^{-2})^2}
-
\frac{
\lbrack \tilde G_2,  W_0\rbrack_q}{(q^2-q^{-2})^2}
-
\frac{
\lbrack W_1, \tilde G_3\rbrack_q}{(q^2-q^{-2})^2},
\\
W_{-4} &=  W_0 -
\frac{
\lbrack W_1, \tilde G_1\rbrack_q}{(q^2-q^{-2})^2}
-
\frac{
\lbrack \tilde G_2, W_0\rbrack_q}{(q^2-q^{-2})^2}
-
\frac{
\lbrack W_1, \tilde G_3\rbrack_q}{(q^2-q^{-2})^2}
-
\frac{
\lbrack\tilde G_4,W_0\rbrack_q}{(q^2-q^{-2})^2},
\\
 W_{6} &=  W_0 -
\frac{
\lbrack W_1, \tilde G_1\rbrack_q}{(q^2-q^{-2})^2}
-
\frac{
\lbrack \tilde G_2,  W_0\rbrack_q}{(q^2-q^{-2})^2}
-
\frac{
\lbrack  W_1, \tilde G_3\rbrack_q}{(q^2-q^{-2})^2}
-
\frac{
\lbrack \tilde G_4, W_0\rbrack_q}{(q^2-q^{-2})^2}
-
\frac{
\lbrack W_1, \tilde G_5\rbrack_q}{(q^2-q^{-2})^2},
\\
&\cdots
\end{align*}

\noindent 
The recursion shows that for any integer $k \geq 1$, 
the generators $ W_{-k}$, $ W_{k+1}$
are given as follows. For odd $k=2r+1$,
\begin{align}
W_{-k} =  W_1
- 
\sum_{\ell=0}^r 
\frac{
\lbrack 
 \tilde G_{2\ell+1},  W_0\rbrack_q
}{(q^2-q^{-2})^2} 
-
\sum_{\ell=1}^r 
\frac{
\lbrack
 W_1,
 \tilde G_{2\ell} 
 \rbrack_q
}{(q^2-q^{-2})^2},
\label{eq:WmkA}
\\
W_{k+1} =  W_0
- 
\sum_{\ell=0}^r 
\frac{
\lbrack 
W_1,
 {\tilde G}_{2\ell+1}\rbrack_q
}{(q^2-q^{-2})^2} 
-
\sum_{\ell=1}^r 
\frac{
\lbrack
 {\tilde G}_{2\ell},
 W_0
 \rbrack_q
}{(q^2-q^{-2})^2}.
\label{eq:Wkp1A}
\end{align}
For even $k=2r$,
\begin{align}
&W_{-k} = W_0
- 
\sum_{\ell=0}^{r-1} 
\frac{
\lbrack  W_1,
  {\tilde G}_{2\ell+1}\rbrack_q
}{(q^2-q^{-2})^2} 
-
\sum_{\ell=1}^r 
\frac{
\lbrack
  {\tilde G}_{2\ell},
 W_0
 \rbrack_q
}{(q^2-q^{-2})^2},
\label{eq:WmkB}
\\
\label{eq:Wkp1B}
 &W_{k+1} =  W_1
- 
\sum_{\ell=0}^{r-1} 
\frac{
\lbrack 
{\tilde G}_{2\ell+1}, W_0\rbrack_q
}{(q^2-q^{-2})^2} 
-
\sum_{\ell=1}^r 
\frac{
\lbrack
 W_1,
 {\tilde G}_{2\ell} 
 \rbrack_q
}{(q^2-q^{-2})^2}. 
\end{align}
Next we use 
\eqref{eq:3p1} to
obtain the generators
$\lbrace G_{k+1}\rbrace_{k \in \mathbb N}$:
\begin{align}
G_{k+1} =  {\tilde G}_{k+1} + (q+q^{-1}) 
\lbrack W_1,  W_{-k} \rbrack
\qquad \qquad (k \in \mathbb N).
\label{eq:getrid}
\end{align}
We have expressed the elements \eqref{eq:4gens} in terms of
$W_0$, $W_1$, $\lbrace \tilde G_{k+1} \rbrace_{k \in \mathbb N}$. Next, we would like to know how the elements $\lbrace \tilde G_{k+1} \rbrace_{k\in \mathbb N}$ are related to the elements \eqref{eq:Upbw}. 
We will discuss this relationship using generating functions.
\medskip

\noindent Recall the generating function $B(t)$ from Definition \ref{def:Bgen}. The generating function $B(t)$ is commutative by Definition
\ref{def:com} and
 the comment above Lemma \ref{prop:damiani}.
 By \eqref{eq:zerodelta} the generating function $B(t)$ has constant term $q^{-2}-1 = -q^{-1}(q-q^{-1})$, so by 
\eqref{eq:avee} we have
\begin{align*}
B(t)^\vee  = -q(q-q^{-1})^{-1} B(t).
%\label{eq:Bnorm}
\end{align*}
The generating function $B(t)^\vee$ is commutative and normalized, so we may speak of its $q$-expansion as is Definition \ref{def:qexpC}.

\begin{conjecture} \label{conj:M} \rm
Define the elements 
\begin{align}
\label{eq:4gensc}
 \lbrace W_{-k}\rbrace_{k\in \mathbb N}, \qquad \lbrace W_{k+1}\rbrace_{k \in \mathbb N}, \qquad  
 \lbrace G_{k+1} \rbrace_{k \in \mathbb N},
\qquad
\lbrace  \tilde G_{k+1} \rbrace_{k \in \mathbb N}
\end{align}
in $\mathcal O_q$ as follows:
\begin{enumerate}
\item[\rm (i)] the generating function $\tilde G(t)^\vee$ is the 
$q$-expansion of $B(t)^\vee$, where  $\tilde G(t) = \sum_{n \in \mathbb N} \tilde G_n t^n$ 
and $\tilde G_0$ is from \eqref{eq:GtG};
\item[\rm (ii)]  the elements $\lbrace W_{-k}\rbrace_{k\in \mathbb N}$,
$\lbrace W_{k+1} \rbrace_{k\in \mathbb N}$ satisfy
\eqref{eq:WmkA}--\eqref{eq:Wkp1B};
\item[\rm (iii)] the elements
$\lbrace G_{k+1}\rbrace_{k \in \mathbb N}$ satisfy \eqref{eq:getrid}.
\end{enumerate}
Then the elements \eqref{eq:4gensc} satisfy \eqref{eq:3p1}--\eqref{eq:3p11}.
\end{conjecture}

\noindent  We have some comments about the $q$-expansion of $B(t)^\vee$. 
 We mentioned above that $B(t)$ is commutative, so by Lemma
\ref{lem:AB} there exists a commutative subalgebra $\mathcal A$ of $\mathcal O_q$ that contains $B_{n \delta}$ for $n \in \mathbb N$.
So $B(t)$ is over $\mathcal A$.
The $q$-expansion of $B(t)^\vee$ is over $\mathcal A$, and described as follows.
For the moment let $\tilde G(t)= \sum_{n \in \mathbb N} \tilde G_n t^n$ denote any generating function over $\mathcal A$ such that  $\tilde G_0$ satisfies
\eqref{eq:GtG}. By Proposition \ref{prop:combine} and Definitions \ref{def:qexp}, \ref{def:qexpC} we find that
\begin{align*}
\hbox{\rm
$\tilde G(t)^\vee$ is the 
$q$-expansion of $B(t)^\vee$}
\end{align*}
\noindent if and only if 
\begin{align}
\label{eq:conjMain}
B(t) 
{\tilde G}\biggl( \frac{q+q^{-1}}{qt+q^{-1} t^{-1}}\biggr)
{\tilde G}\biggl( \frac{q+q^{-1}}{q^{-1}t+q t^{-1}}\biggr) = -q^{-1} (q-q^{-1})^3 \lbrack 2\rbrack_q^4
\end{align}
 if and only if 
\begin{align}
\label{eq:genfun}
B(qt) 
{\tilde G}\biggl( \frac{q+q^{-1}}{q^2t+q^{-2} t^{-1}}\biggr) = B(q^{-1}t) 
{\tilde G}\biggl( \frac{q+q^{-1}}{q^{-2}t+q t^{-2}}\biggr)
\end{align}
 if and only if for $n\geq 1$,
\begin{align}
\label{eq:recBG}
0 = \lbrack n \rbrack_q B_{n\delta}  {\tilde G}_0 
+
\sum_{\stackrel{\scriptstyle j+k+2\ell+1=n,}{\scriptstyle j,k,\ell\geq 0}}
(-1)^\ell 
\binom{k+\ell}{\ell}
\lbrack 2n-j \rbrack_q
\lbrack 2 \rbrack^{k+1}_q
B_{j\delta}  {\tilde G}_{k+1}.
\end{align}
In Appendix A we display \eqref{eq:recBG} in detail for $1\leq n\leq 8$.
%%%E_i V \circ E_j V = \sum_{\stackrel{ \scriptstyle 0 \leq h \leq d }{ \scriptstyle q^h_{ij} \not=0}} E_hV

\section{Supporting evidence for Conjecture \ref{conj:M}}
 
\noindent In this section we give some supporting evidence for Conjecture \ref{conj:M}.
\medskip

\noindent Our first type of evidence is from checking via computer. The algebra $\mathcal O_q$ 
has been implemented in the computer package SageMath (see \cite{sage}) by Travis Scimshaw. Using this package Scrimshaw defined the elements
\eqref{eq:4gensc} for $0 \leq k \leq 5$ using \eqref{eq:recBG} along with \eqref{eq:WmkA}--\eqref{eq:Wkp1B} and
 \eqref{eq:getrid}.
 He then had SageMath  verify the relations among
\eqref{eq:3p1}--\eqref{eq:3p11}
that involved these defined elements.

\medskip
\noindent Our next type of evidence concerns the analog of Conjecture \ref{conj:M} for the Onsager algebra $\mathcal O$.
Consider the equation
\eqref{eq:conjMain}. In that equation
we compute the limit $q\mapsto 1$ in two steps: (i) make a change of variables as before; (ii) simplify the result and set $q=1$.\\
\noindent Step (i):
We express our generating functions as 
\begin{align}
\label{eq:Be}
B(t) &= q^{-2}-1 + 4 \xi^2 \mathcal B(t), \qquad \qquad \mathcal B(t) = \sum_{n=1}^\infty B_n t^n,
\\
\label{eq:G3}
\tilde G(t) &= -(q-q^{-1})\lbrack 2 \rbrack^2_q + \xi^2 \tilde G'(t), \qquad \qquad \tilde G'(t) = \sum_{n=1}^\infty \tilde G'_n t^n.
\end{align}
Evaluating \eqref{eq:conjMain}  using \eqref{eq:Be}, \eqref{eq:G3} and $\xi^2 = -(q-q^{-1})^2/4$ we obtain
\begin{align*}
&\biggl(
q^{-2}-1 -(q-q^{-1})^2 \mathcal B(t) \biggr)
\Biggl(
-(q-q^{-1}) \lbrack 2 \rbrack^2_q - \frac{(q-q^{-1})^2}{4} 
{\tilde G'}\biggl( \frac{q+q^{-1}}{qt+q^{-1} t^{-1}}\biggr) \Biggr)
\\
&\times \;
\Biggl(
-(q-q^{-1}) \lbrack 2 \rbrack^2_q - \frac{(q-q^{-1})^2}{4} 
{\tilde G'}\biggl( \frac{q+q^{-1}}{q^{-1}t+qt^{-1}}\biggr) \Biggr)
=
 -q^{-1} (q-q^{-1})^3 \lbrack 2\rbrack_q^4.
 \end{align*}
\noindent Step (ii): For the above equation, let $D$ denote the left-hand side minus the right-hand side. After expanding $D$ and doing some cancelation, we find that
 $D$ is equal to
$-(q-q^{-1})^4 \lbrack 2 \rbrack^2_q/2 $ times
\begin{align}
\label{eq:lead}
2 \lbrack 2 \rbrack^2_q \mathcal B(t) + \frac{1}{2q} {\tilde G'}\biggl( \frac{q+q^{-1}}{qt+q^{-1} t^{-1}}\biggr) 
+ \frac{1}{2q} {\tilde G'}\biggl( \frac{q+q^{-1}}{q^{-1}t+q t^{-1}}\biggr) 
\end{align}
\noindent plus $(q-q^{-1})^5$ times some additional terms. 
Dividing $D$ by $(q-q^{-1})^4$ and then setting $q=1$,  we find that \eqref{eq:conjMain} becomes
\begin{align}
\label{eq:MainOne}
8 \mathcal B(t) + {\tilde G'}\biggl( \frac{2}{t+ t^{-1}}\biggr) =0.
\end{align}
\noindent Equation \eqref{eq:MainOne} 
matches the equation on the right in \cite[Line~(4.8)]{BC17}. By that citation the
equation \eqref{eq:MainOne} is satisfied by the basis for $\mathcal O$ described in Remark \ref{rem:qEqual1C}.
 We have verified the analog of Conjecture \ref{conj:M} that applies to $\mathcal O$.
 \medskip

\noindent Our next type of evidence has to do with the universal Askey-Wilson algebra
 $\Delta_q$ 
\cite[Definition~1.2]{uaw}.
This algebra is defined by generators and relations.
The generators are $A,B,C$. The relations assert that
each of the following is central in $\Delta_q$:
\begin{align*}
A + \frac{qBC-q^{-1}CB}{q^2-q^{-2}},
\qquad \quad 
B + \frac{qCA-q^{-1}AC}{q^2-q^{-2}},
\qquad \quad 
C + \frac{qAB-q^{-1}BA}{q^2-q^{-2}}.
\end{align*}
For the above three central elements, multiply each by
$q+q^{-1}$ to get 
$\alpha$, $\beta$, $\gamma $. Thus
\begin{align}
&A + \frac{qBC-q^{-1}CB}{q^2-q^{-2}} = \frac{\alpha}{q+q^{-1}},
\label{eq:alpha}
\\
&B + \frac{qCA-q^{-1}AC}{q^2-q^{-2}} =\frac{\beta}{q+q^{-1}},
\label{eq:beta}
\\
&C + \frac{qAB-q^{-1}BA}{q^2-q^{-2}}= \frac{\gamma}{q+q^{-1}}.
\label{eq:gamma}
\end{align}
Each of
$\alpha$, $\beta$, $\gamma$ is central in $\Delta_q$.
By \cite[Corollary~8.3]{uaw}
the center of $\Delta_q$ is generated by 
$\alpha, \beta, \gamma, \Omega$ where
\begin{align}
\Omega = qABC+q^2 A^2 + q^{-2} B^2+ q^2 C^2-q A \alpha
-q^{-1} B \beta - q C \gamma.
\label{eq:Cas}
\end{align}
The element $\Omega$ is called the Casimir element.
By \cite[Theorem~8.2]{uaw} the elements $ \alpha, \beta, \gamma, \Omega$ are algebraically independent. We write
$\mathbb F\lbrack  \alpha, \beta, \gamma, \Omega \rbrack$ for the center of $\Delta_q$.
\medskip

\noindent Next we summarize 
from 
\cite[Section~3]{uaw}
 how
the 
 modular group ${\rm PSL}_2(\mathbb Z)$ acts on $\Delta_q$
as a group of automorphisms.
By  
\cite{RCA} 
the group ${\rm PSL}_2(\mathbb Z)$ has a presentation
by generators $\varrho$, $\sigma$ and relations $\varrho^3=1$, $\sigma^2=1$.
By \cite[Theorems~3.1,~6.4]{uaw} the group
${\rm PSL}_2(\mathbb Z)$ acts on $\Delta_q$ as a group of automorphisms
in the following way:
\begin{align*}
%%%%\label{eq:table}
{\rm
%%\begin{center}
\begin{tabular}{c| ccc | c c c c}
$u$ &  $A$ & $B$  & $C$
&  $\alpha $ & $\beta $  & $\gamma$ &$\Omega$
\\
\hline
$\varrho(u)$ &  $B$ & $C$ & $A$
&  $\beta $ & $\gamma $ & $\alpha$ & $\Omega$
\\
$\sigma(u) $ &  $B$ & $A$ & $C+\frac{\lbrack A,B\rbrack}{q-q^{-1}}$
&  $\beta $ & $\alpha $ & $\gamma$ & $\Omega$
\end{tabular}
 %%       \end{center}
}
\end{align*}

%%%\noindent This action is faithful by CI\cite[Theorem 3.13]{uaw}.
\noindent For notational convenience define
\begin{align}
\label{eq:Cp}
C' = C + \frac{\lbrack A,B\rbrack}{q-q^{-1}}.
\end{align}
\noindent Applying $\sigma$ to \eqref{eq:alpha}--\eqref{eq:gamma} and using the above table, we obtain
\begin{align}
&B + \frac{qAC'-q^{-1}C'A}{q^2-q^{-2}} = \frac{\beta}{q+q^{-1}},
\label{eq:alphaP}
\\
&A+ \frac{qC'B-q^{-1}BC'}{q^2-q^{-2}} =\frac{\alpha}{q+q^{-1}},
\label{eq:betaP}
\\
&C' + \frac{qBA-q^{-1}AB}{q^2-q^{-2}}= \frac{\gamma}{q+q^{-1}}.
\label{eq:gammaP}
\end{align}

\noindent Next we explain how $\Delta_q$ is related to
$\mathcal O_q$. 
By \cite[Theorem~2.2]{uaw}
 the algebra $\Delta_q$
has a presentation by generators $A,B,\gamma$ and relations
\begin{align}
&A^3B-\lbrack 3\rbrack_q A^2BA+\lbrack 3\rbrack_q ABA^2-BA^3
= (q^2-q^{-2})^2(BA-AB),
\label{eq:rel1}
\\
&B^3A-\lbrack 3\rbrack_q B^2AB+\lbrack 3\rbrack_q BAB^2-AB^3
= (q^2-q^{-2})^2(AB-BA),
\label{eq:rel2}
\\
&A^2B^2-B^2A^2+(q^2+q^{-2})(BABA-ABAB) = (q-q^{-1})^2(BA-AB)\gamma,
\label{eq:rel3}
\\
& \qquad \qquad \gamma A = A \gamma, \qquad \qquad \gamma B=B \gamma.
\label{eq:rel4}
 \end{align}
\noindent 
The relations
\eqref{eq:rel1}, 
\eqref{eq:rel2}  are the $q$-Dolan/Grady relations.
Consequently
there exists an 
 algebra homomorphism $\natural :\mathcal O_q \to \Delta_q$
that sends $W_0\mapsto A$ and $W_1\mapsto B$.
This homomorphism is not injective
by \cite[Theorem~10.9]{uaw}.
\medskip

\noindent For the elements   \eqref{eq:Upbw} and \eqref{eq:4gensc}
we retain the same notation for their images under $\natural$. We will show that for $\Delta_q$ the elements \eqref{eq:4gensc} satisfy the relations \eqref{eq:3p1}--\eqref{eq:3p11}.
 \medskip
 
\noindent 
 For the algebra $\Delta_q$ define
\begin{align}
\label{eq:chvar}
\Psi (t) = B(t)+1-q^{-2},
\end{align}
where $B(t)$ is from
Definition \ref{def:Bgen}.
 %\label{eq:zerodelta}
By \eqref{eq:zerodelta} we have $\Psi(t)= \sum_{n=1}^\infty B_{n\delta} t^n$. By \cite[Corollary~5.7]{pbw} 
the elements $\lbrace B_{n\delta}\rbrace_{n=1}^\infty $are contained in the subalgebra of $\Delta_q$ generated by $\mathbb F\lbrack \alpha, \beta, \gamma, \Omega \rbrack$ and $C$.
 Consequently the elements  $\lbrace B_{n\delta}\rbrace_{n=1}^\infty$ commute with $C$, so $\Psi(t)$ commutes with $C$. By this and \cite[Line (5.19)]{pbw} 
we find that 
\begin{align}
\Psi(t) \bigl(q t + q^{-1} t^{-1} + C\bigr)
\bigl(q^{-1} t + q t^{-1} + C\bigr)
\label{eq:3}
\end{align}
is equal to $1-q^{-2}$ times
\begin{align*}
 \Omega
- \frac{(t+t^{-1}) \alpha \beta }{(t-t^{-1})^2}
-
\frac{\alpha^2+ \beta^2}{(t-t^{-1})^2}
- (t+t^{-1})\gamma
+(q+q^{-1})(t+t^{-1}) C + C^2.
\end{align*}
Upon eliminating $\Psi(t)$ from \eqref{eq:3} using \eqref{eq:chvar}, we find that
\begin{align}
B(t) \bigl(q t + q^{-1} t^{-1} + C\bigr)
\bigl(q^{-1} t + q t^{-1} + C\bigr)
\label{eq:3B}
\end{align}
is equal to $1-q^{-2}$ times
\begin{align*}
 \Omega
- \frac{(t+t^{-1}) \alpha \beta }{(t-t^{-1})^2}
-
\frac{\alpha^2+ \beta^2}{(t-t^{-1})^2}
- (t+t^{-1})\gamma
-(qt+q^{-1}t^{-1})(q^{-1}t+qt^{-1}).
\end{align*}
%By this and
%\begin{align*}
%(t-t^{-1})^{-2} = t^2 (1-t^2)^{-2} =  \sum_{n=1}^\infty n t^{2n}
%\end{align*}
\noindent Define
\begin{align}
\label{eq:N}
N(t) = \frac{B(t)}{q^{-2}-1}\,
\frac{qt+q^{-1} t^{-1} + C}{qt+q^{-1}t^{-1}} \,
\frac{q^{-1}t+qt^{-1} + C}{q^{-1} t+ q t^{-1}}.
\end{align}
\noindent By the above comments
\begin{align}
N(t) = 1 + N_1(t)\Omega +  N_2(t)\alpha \beta  + N_3(t) (\alpha^2+\beta^2) + N_4(t) \gamma,
\label{eq:NNNN}
\end{align}
where
\begin{align}
N_1(t) &= \frac{-1}{ (qt+q^{-1} t^{-1})(q^{-1}t + q t^{-1})},
\label{eq:N1}
\\
N_2(t) &= \frac{t+t^{-1}}{(t-t^{-1})^2 (qt+q^{-1} t^{-1})(q^{-1}t + q t^{-1})},
\label{eq:N2}
\\
N_3(t) &= \frac{1}{(t-t^{-1})^2 (qt+q^{-1} t^{-1})(q^{-1}t + q t^{-1})},
\label{eq:N3}
\\
N_4(t) &= \frac{t+t^{-1}}{(qt+q^{-1} t^{-1})(q^{-1}t + q t^{-1})}.
\label{eq:N4}
\end{align}
Evaluating \eqref{eq:N1}--\eqref{eq:N4} using
\begin{align*}
\frac{1}{qt+q^{-1} t^{-1}} &= \sum_{n \in \mathbb N} (-1)^n q^{2n+1} t^{2n+1},
\\
\frac{1}{q^{-1}t+q t^{-1}} &= \sum_{n \in \mathbb N} (-1)^n q^{-2n-1} t^{2n+1},
\\
\frac{1}{(t-t^{-1})^{2}}&= \sum_{n\in \mathbb N} n t^{2n}
\end{align*}
we find that the functions $N_1(t)$, $N_2(t)$, $N_3(t)$, $N_4(t)$ are power series in $t$ with zero constant term.
By this and \eqref{eq:NNNN}, we may view
 $N(t)$ as a normalized generating function over $\mathbb F \lbrack \alpha, \beta, \gamma, \Omega\rbrack$.
 
 \begin{definition} \label{def:Z} \rm
 Define a generating function $Z(t)=\sum_{n \in \mathbb N} Z_n t^n$ over  $\mathbb F \lbrack \alpha, \beta, \gamma, \Omega\rbrack$ such that $Z_0=q^{-2}-q^2$
 and $Z(t)^\vee$ is the $q$-expansion of $N(t)$. 
 \end{definition}
 \noindent  The notation $Z(t)^\vee$ is explained in \eqref{eq:avee}. The $q$-expansion concept is explained in Proposition
 \ref{prop:combine} and
 Definition
  \ref{def:qexp}. By these explanations and
 Definition \ref{def:Z},
 \begin{align}
 N(t) Z\biggl(\frac{q+q^{-1}}{qt+q^{-t}t^{-1}} \biggr) Z\biggl(\frac{q+q^{-1}}{q^{-1} t + q t^{-1}}\biggr) = (q^2-q^{-2})^2.
 \label{eq:NZZ}
 \end{align}

\begin{proposition} \label{prop:th}
For the algebra $\Delta_q$,
\begin{align}
\tilde G(t) = Z(t) (q+q^{-1} + t C).
\label{eq:GC}
\end{align}
\end{proposition}
\begin{proof} Define the generating function $\mathcal {\tilde G}(t) = Z(t)(q+q^{-1} + tC)$. We show that $\tilde G(t) = \mathcal {\tilde G}(t)$.
Let $\mathcal A$ denote the subalgebra of $\Delta_q$ generated by $\mathbb F\lbrack \alpha, \beta, \gamma, \Omega \rbrack$ and $C$. Note that $\mathcal A$ is commutative.
By construction $\mathcal {\tilde G}(t)$ is over $\mathcal A$.
 By our comments below \eqref{eq:chvar},
the generating function $B(t)$ is over $\mathcal A$.
By the discussion around \eqref{eq:conjMain}, it suffices to show that
\begin{align}
\label{eq:conjMainT}
B(t) 
\mathcal {\tilde G}\biggl( \frac{q+q^{-1}}{qt+q^{-1} t^{-1}}\biggr)
\mathcal {\tilde G}\biggl( \frac{q+q^{-1}}{q^{-1}t+q t^{-1}}\biggr) = -q^{-1} (q-q^{-1})^3 \lbrack 2\rbrack_q^4.
\end{align}
Using 
\eqref{eq:N} and
 and \eqref{eq:NZZ},
 \begin{align*}
B(t) &
\mathcal {\tilde G}\biggl( \frac{q+q^{-1}}{qt+q^{-1} t^{-1}}\biggr)
\mathcal {\tilde G}\biggl( \frac{q+q^{-1}}{q^{-1}t+q t^{-1}}\biggr) 
\\&= \lbrack 2 \rbrack^2_q  
B(t) Z\biggl( \frac{q+q^{-1}}{qt+q^{-1} t^{-1}}\biggr)
\frac{qt+q^{-1} t^{-1} + C}{qt+q^{-1}t^{-1}} 
Z\biggl( \frac{q+q^{-1}}{q^{-1}t+q t^{-1}}\biggr) 
\frac{q^{-1}t+qt^{-1} + C}{q^{-1} t+ q t^{-1}} 
\\
&= \lbrack 2 \rbrack^2_q (q^{-2}-1) N(t) 
 Z\biggl( \frac{q+q^{-1}}{qt+q^{-1} t^{-1}}\biggr)
Z\biggl( \frac{q+q^{-1}}{q^{-1}t+q t^{-1}}\biggr) 
\\
&=
\lbrack 2 \rbrack^2_q (q^{-2}-1) (q^2-q^{-2})^2 
\\
&= -q^{-1}(q-q^{-1})^3 \lbrack 2 \rbrack^4_q.
\end{align*}
We have shown \eqref{eq:conjMainT}, and the result follows.
\end{proof}

\noindent Define the generating functions
\begin{align*}
W^-(t) = \sum_{n\in \mathbb N} W_{-n} t^n, \qquad\qquad W^+(t) = \sum_{n \in \mathbb N} W_{n+1} t^n.
%%%% \qquad \quad G(t) = \sum_{n\in \mathbb N} G_n t^n.
\end{align*}
By \eqref{eq:WmkA}--\eqref{eq:Wkp1B} we obtain
\begin{align}
&W^+(t) = \frac{t 
\lbrack \tilde G(t), A\rbrack_q + 
\lbrack B, \tilde G(t) \rbrack_q}{ 
(t^2-1)(q^2-q^{-2})^2},
\label{eq:wpP}
\\
&W^-(t) = \frac{ 
\lbrack \tilde G(t), A\rbrack_q + 
t \lbrack B, \tilde G(t) \rbrack_q}{ 
(t^2-1)(q^2-q^{-2})^2}.
\label{eq:wmP}
\end{align}

\begin{lemma} 
\label{lem:WWGgf}
For the algebra $\Delta_q$,
\begin{align}
\label{eq:wp}
W^+(t) &= Z(t) \,\frac{ (q-q^{-1})(\alpha+\beta t) -(q^2-q^{-2})(t-t^{-1})B}{(q^2-q^{-2})^2(t-t^{-1})},
\\
\label{eq:wm}
W^-(t) &= Z(t) \, \frac{ (q-q^{-1}) (\alpha t + \beta) - (q^2-q^{-2})(t-t^{-1}) A}{ (q^2-q^{-2})^2(t-t^{-1})}.
\end{align}
\end{lemma}
\begin{proof} To obtain \eqref{eq:wp}, eliminate $\tilde G(t)$ from \eqref{eq:wpP} using \eqref{eq:GC}, and evaluate the result using
\eqref{eq:alpha}, \eqref{eq:beta}. Equation \eqref{eq:wm} is similarly obtained.
\end{proof}

\noindent Define the generating function
\begin{align*}
G(t) = \sum_{n\in \mathbb N} G_n t^n.
\end{align*}
Using  \eqref{eq:getrid} we obtain
\begin{align}
G(t) = \tilde G(t) + t(q+q^{-1}) \lbrack B, W^-(t)\rbrack.
\label{eq:GP}
\end{align}
\begin{lemma} \label{lem:Ggf} For the algebra $\Delta_q$ we have
\begin{align}
G(t) = Z(t) (q+ q^{-1} + t C'),
\label{eq:Gform}
\end{align}
where $C'$ is from  \eqref{eq:Cp}.
\end{lemma}
\begin{proof} Eliminate $\tilde G(t)$ from \eqref{eq:GP} using
\eqref{eq:GC}.
Eliminate $W^-(t)$ from \eqref{eq:GP} using
\eqref{eq:wm}, and evaluate the result using \eqref{eq:Cp}.
\end{proof}

\noindent Let $s$ denote an indeterminate that commutes with $t$.
\begin{lemma} \label{lem:ad} For the algebra $\Delta_q$ we have
\begin{align*}
&
\lbrack A, W^+(t) \rbrack = \lbrack W^-(t), B \rbrack = t^{-1}(\tilde G(t)-G(t))/(q+q^{-1}),
\\
& 
\lbrack A, G(t) \rbrack_q = \lbrack \tilde G(t), A \rbrack_q = \rho W^-(t)-\rho t W^+(t),
\\
&
\lbrack G(t), B \rbrack_q = \lbrack B, \tilde G(t) \rbrack_q = \rho W^+(t) -\rho t W^-(t),
\\
&
\lbrack  W^-(s), W^-(t) \rbrack = 0, 
\qquad 
\lbrack W^+(s),  W^+(t) \rbrack = 0,
\\
&\lbrack  W^-(s), W^+(t) \rbrack 
+
\lbrack W^+(s), W^-(t) \rbrack = 0,
\\
&s \lbrack  W^-(s), G(t) \rbrack 
+
t \lbrack  G(s),  W^-(t) \rbrack = 0,
\\
&s \lbrack  W^-(s), {\tilde G}(t) \rbrack 
+
t \lbrack  \tilde G(s),  W^-(t) \rbrack = 0,
\\
&s \lbrack   W^+(s),  G(t) \rbrack
+
t \lbrack   G(s), W^+(t) \rbrack = 0,
\\
&s \lbrack   W^+(s), {\tilde G}(t) \rbrack
+
t \lbrack \tilde G(s), W^+(t) \rbrack = 0,
\\
&\lbrack   G(s), G(t) \rbrack = 0, 
\qquad 
\lbrack  {\tilde G}(s),  {\tilde G}(t) \rbrack = 0,
\\
&\lbrack   {\tilde G}(s), G(t) \rbrack +
\lbrack   G(s), {\tilde G}(t) \rbrack = 0,
\end{align*}
where $\rho=-(q^2-q^{-2})^2$.
\end{lemma}
\begin{proof} These relations are routinely verified using
Proposition \ref{prop:th} and
Lemmas \ref{lem:WWGgf}, \ref{lem:Ggf} along with
\eqref{eq:alpha}, \eqref{eq:beta}, \eqref{eq:alphaP}, \eqref{eq:betaP}.
%% and (\ref{eq:GP}).
\end{proof}

\begin{theorem} In the algebra $\Delta_q$ the
elements \eqref{eq:4gensc} satisfy the relations \eqref{eq:3p1}--\eqref{eq:3p11}.
\end{theorem}
\begin{proof} This is a routine consequence of  Lemma
 \ref{lem:ad}.
\end{proof}

\section{Comments}

\noindent In the previous section we gave some supporting evidence for Conjecture \ref{conj:M}. In this section we assume that Conjecture \ref{conj:M} is correct,
and provide more information about how the elements \eqref{eq:4gensc} are related to the elements \eqref{eq:Upbw}. We will give a variation on \eqref{eq:WmkA}--\eqref{eq:Wkp1B}. 
\medskip

\noindent Using Appendix A and $B_{\delta} = q^{-2}W_1W_0 - W_0 W_1$ we obtain
\begin{align}
{\tilde G}_1 = -q B_\delta = \lbrack W_0, W_1\rbrack_q.
\label{eq:G1}
\end{align}

\begin{lemma} \label{lem:GWcom} For $k \in \mathbb N$,
\begin{enumerate}
\item[\rm (i)] $\lbrack {\tilde G}_{k+1}, W_0\rbrack_q = (q-q^{-1})W_0 {\tilde G}_{k+1} -q^2 \lbrack B_\delta, W_{-k} \rbrack$,
\item[\rm (ii)] $\lbrack W_1, {\tilde G}_{k+1} \rbrack_q = (q-q^{-1})W_1 {\tilde G}_{k+1} +\lbrack B_\delta, W_{k+1} \rbrack$.
\end{enumerate}
\end{lemma}
\begin{proof} (i) Observe that
\begin{align*}
\lbrack {\tilde G}_{k+1}, W_0 \rbrack_q = (q-q^{-1}) W_0 {\tilde G}_{k+1}+q \lbrack {\tilde G}_{k+1}, W_0\rbrack.
\end{align*}
By \eqref{eq:3p7} and \eqref{eq:G1},
\begin{align*}
\lbrack {\tilde G}_{k+1}, W_0\rbrack = \lbrack {\tilde G}_1, W_{-k} \rbrack = -q \lbrack B_\delta, W_{-k} \rbrack.
\end{align*}
\noindent The result follows.
\\
\noindent (ii) Observe that
\begin{align*}
\lbrack W_1, {\tilde G}_{k+1} \rbrack_q = (q-q^{-1}) W_1 {\tilde G}_{k+1}-q^{-1} \lbrack {\tilde G}_{k+1}, W_1\rbrack.
\end{align*}
By \eqref{eq:3p9} and \eqref{eq:G1},
\begin{align*}
\lbrack {\tilde G}_{k+1}, W_1\rbrack = \lbrack {\tilde G}_1, W_{k+1} \rbrack = -q \lbrack B_\delta, W_{k+1} \rbrack.
\end{align*}
\noindent The result follows.
\end{proof}
\begin{lemma}\label{lem:Wind}
For $n\geq 1$,
\begin{align}
\label{eq:Wind1}
W_{-n} &= W_n -\frac{(q-q^{-1}) W_0 {\tilde G}_n}{(q^2-q^{-2})^2} + \frac{q^2 \lbrack B_\delta, W_{1-n}\rbrack}{(q^2-q^{-2})^2},
\\
W_{n+1} &=W_{1-n}-\frac{(q-q^{-1}) W_1 {\tilde G}_n}{(q^2-q^{-2})^2} -\frac{\lbrack B_\delta, W_n\rbrack}{(q^2-q^{-2})^2}.
\label{eq:Wind2}
\end{align}
\end{lemma}
\begin{proof} Use the equations on the right in
\eqref{eq:3p2}, \eqref{eq:3p3} along with Lemma  \ref{lem:GWcom}.
\end{proof}
\noindent 
We recall some notation from \cite{BK}. For a negative integer $k$ define
\begin{align*}
B_{k \delta + \alpha_0} = B_{(-k-1)\delta + \alpha_1}, \qquad \qquad 
B_{k\delta+\alpha_1} = B_{(-k-1)\delta+\alpha_0}.
\end{align*}
We have
\begin{align}
 B_{r\delta+\alpha_0} = B_{s\delta+\alpha_1} \qquad \qquad (r,s \in \mathbb Z, \quad r+s=-1).
 \label{eq:extend}
 \end{align}

\begin{lemma} For $n\in \mathbb Z$,
\begin{align}
\label{eq:nnot1}
\frac{q\lbrack B_\delta, B_{n\delta+\alpha_0} \rbrack}{(q-q^{-1})(q^2-q^{-2})}&= B_{(n+1) \delta +\alpha_0} - B_{(n-1)\delta +\alpha_0},
\\
\label{eq:nnot2}
\frac{q\lbrack B_\delta, B_{n\delta+\alpha_1} \rbrack}{(q-q^{-1})(q^2-q^{-2})}&= B_{(n-1)\delta +\alpha_1} - B_{(n+1)\delta +\alpha_1}.
\end{align}
\end{lemma}
\begin{proof} Use \eqref{eq:line1}--\eqref{eq:line4} and \eqref{eq:extend}.
\end{proof}

\begin{proposition}
\label{prop:WWalt} For $n \in \mathbb N$ the following hold in $\mathcal O_q$:
\begin{align}
\label{eq:recWnm1}
W_{-n} &=-(q-q^{-1})^{-1} \sum_{k=0}^n \sum_{\ell=0}^k \binom{k}{\ell} q^{k-2 \ell} \lbrack 2 \rbrack^{-k-2}_q B_{(k-2\ell)\delta + \alpha_0} {\tilde G}_{n-k},
\\
W_{n+1} &=  -(q-q^{-1})^{-1}  \sum_{k=0}^n \sum_{\ell=0}^k \binom{k}{\ell} q^{2 \ell-k} \lbrack 2 \rbrack^{-k-2}_q B_{(k-2\ell)\delta + \alpha_1} {\tilde G}_{n-k}.
\label{eq:recWnm2}
\end{align}
\end{proposition}
\begin{proof} We use induction on $n$.
First assume that $n=0$. Then \eqref{eq:recWnm1}, \eqref{eq:recWnm2} hold. Next assume that $n\geq 1$.
To obtain \eqref{eq:recWnm1}, evaluate the right-hand side of 
 \eqref{eq:Wind1} 
 using induction along with \eqref{eq:extend}, \eqref{eq:nnot1}.
 To obtain \eqref{eq:recWnm2}, evaluate the right-hand side of 
 \eqref{eq:Wind2}  
  using induction along with \eqref{eq:extend}, \eqref{eq:nnot2}.
\end{proof}
\noindent In Appendix B we display \eqref{eq:recWnm1}, \eqref{eq:recWnm2} in detail for $0 \leq n \leq 7$.
\medskip

\noindent Referring to \eqref{eq:recWnm1} and \eqref{eq:recWnm2}, if we express each term ${\tilde G}_{n-k}$ as a polynomial in $B_\delta, B_{2\delta}, \ldots, B_{(n-k)\delta}$ using \eqref{eq:recBG},
then we effectively write $W_{-n}$ and $W_{n+1}$ in the PBW basis for $\mathcal O_q$ given in Lemma \ref{prop:damiani}.
Unfortunately the resulting formula are not pleasant.

\section{Acknowledgment} The author is deeply grateful to Travis Scrimshaw for performing the computer checks mentioned at the beginning of Section 7. The author thanks Pascal Baseilhac and Nicolas Cramp{\'e} for
giving this paper a close reading and offering valuable comments. The author thanks the referee for giving detailed instructions about how to improve several aspects of the paper.

\section{Appendix A}
For  the $q$-Onsager algebra $\mathcal O_q$ we use
 \eqref{eq:recBG} to obtain
$\tilde G_1, \tilde G_2, \ldots, \tilde G_8$ in terms of
$B_\delta, B_{2\delta}, \ldots, B_{8\delta}$.
\medskip

\noindent
Recall that
\begin{align*}
%\label{eq:conv}
B_{0 \delta} = q^{-2}-1,
%%B_{0\delta} = -q^{-1}(q-q^{-1}),
\qquad \qquad 
 {\tilde G}_0 = - (q-q^{-1}) \lbrack 2 \rbrack^2_q.
\end{align*}
\noindent $ {\tilde G}_1$ satisfies 
\bigskip

\centerline{
0 = 
\begin{tabular}[t]{c|cc}
    & $\lbrack 2 \rbrack_q B_{0\delta}$ & $\lbrack 1 \rbrack_q B_{1 \delta}$ 
\\
\hline
$ {\tilde G}_0$ & $0$ & $1$
\\
$\lbrack 2 \rbrack_q  {\tilde G}_1$ & $1$ & $0$
\\
\end{tabular}}

\newpage
\noindent $ {\tilde G}_2$ satisfies 
\bigskip

\centerline{
0 = 
\begin{tabular}[t]{c|ccc}
    &
    $\lbrack 4 \rbrack_q B_{0 \delta}$ &
    $\lbrack 3 \rbrack_q B_{1 \delta}$ &
    $\lbrack 2 \rbrack_q B_{2 \delta}$ 
\\
\hline
$ {\tilde G}_0$ & $0$ & $0$ & $1$
\\
$\lbrack 2 \rbrack_q  {\tilde G}_1$ & $0$ & $1$ & $0$
\\
$\lbrack 2 \rbrack^2_q  {\tilde G}_2$ & $1$ & $0$ & $0$
\end{tabular}}

\noindent $ {\tilde G}_3$ satisfies 
\bigskip

\centerline{
0 = 
\begin{tabular}[t]{c|cccc}
    &
    $\lbrack 6 \rbrack_q B_{0 \delta}$ &
    $\lbrack 5 \rbrack_q B_{1 \delta}$ &
    $\lbrack 4 \rbrack_q B_{2 \delta}$ &
    $\lbrack 3 \rbrack_q B_{3 \delta}$ 
\\
\hline
${\tilde G}_0$ & $0$ & $0$ & $0$ & $1$
\\
$\lbrack 2 \rbrack_q  {\tilde G}_1$ & $-1$ & $0$ & $1$ & $0$
\\
$\lbrack 2 \rbrack^2_q  {\tilde G}_2$ & $0$ & $1$ & $0$ & $0$
\\
$\lbrack 2 \rbrack^3_q {\tilde G}_3$ & $1$ & $0$ & $0$ & $0$
\end{tabular}}

\noindent $ {\tilde G}_4$ satisfies 
\bigskip

\centerline{
0 = 
\begin{tabular}[t]{c|ccccc}
    &
    $\lbrack 8 \rbrack_q B_{0 \delta}$ &
    $\lbrack 7 \rbrack_q B_{1 \delta}$ &
    $\lbrack 6 \rbrack_q B_{2 \delta}$ &
    $\lbrack 5 \rbrack_q B_{3 \delta}$ & 
    $\lbrack 4 \rbrack_q B_{4 \delta}$  
\\
\hline
$ {\tilde G}_0$ & $0$ & $0$ & $0$ & $0$ & $1$
\\
$\lbrack 2 \rbrack_q  {\tilde G}_1$ & $0$ & $-1$ & $0$ & $1$ & $0$
\\
$\lbrack 2 \rbrack^2_q {\tilde G}_2$ & $-2$ & $0$ & $1$ & $0$ & $0$
\\
$\lbrack 2 \rbrack^3_q {\tilde G}_3$ & $0$ &  $1$ & $0$ & $0$ & $0$
\\
$\lbrack 2 \rbrack^4_q {\tilde G}_4$ & $1$ & $0$ & $0$ & $0$ & $0$
\end{tabular}}

\bigskip

\noindent $ {\tilde G}_5$ satisfies 
\bigskip

\centerline{
0 = 
\begin{tabular}[t]{c|cccccc}
    &
    $\lbrack 10 \rbrack_q B_{0 \delta}$ &
    $\lbrack 9 \rbrack_q B_{1 \delta}$ &
    $\lbrack 8 \rbrack_q B_{2 \delta}$ &
    $\lbrack 7 \rbrack_q B_{3 \delta}$ & 
    $\lbrack 6 \rbrack_q B_{4 \delta}$ &
    $\lbrack 5 \rbrack_q B_{5 \delta}$  
\\
\hline
${\tilde G}_0$ & $0$&  $0$ & $0$ & $0$ & $0$ & $1$
\\
$\lbrack 2 \rbrack_q  {\tilde G}_1$ & $1$ & $0$ & $-1$ & $0$ & $1$ & $0$
\\
$\lbrack 2 \rbrack^2_q  {\tilde G}_2$ & $0$ & $-2$ & $0$ & $1$ & $0$ & $0$
\\
$\lbrack 2 \rbrack^3_q {\tilde G}_3$ & $-3$ & $0$ &  $1$ & $0$ & $0$ & $0$
\\
$\lbrack 2 \rbrack^4_q {\tilde G}_4$ &$0$ & $1$ & $0$ & $0$ & $0$ & $0$
\\
$\lbrack 2 \rbrack^5_q {\tilde G}_5$ & $1$ & $0$&  $0$ & $0$ & $0$ & $0$
\end{tabular}}

\bigskip

\noindent $ {\tilde G}_6$ satisfies 
\bigskip

\centerline{
0 = 
\begin{tabular}[t]{c|ccccccc}
    &
    $\lbrack 12 \rbrack_q B_{0 \delta}$ &
    $\lbrack 11 \rbrack_q B_{1 \delta}$ &
    $\lbrack 10 \rbrack_q B_{2 \delta}$ &
    $\lbrack 9 \rbrack_q B_{3 \delta}$ & 
    $\lbrack 8 \rbrack_q B_{4 \delta}$ &
    $\lbrack 7 \rbrack_q B_{5 \delta}$ &
    $\lbrack 6 \rbrack_q B_{6 \delta}$  
\\
\hline
$ {\tilde G}_0$ & $0$& $0$&  $0$ & $0$ & $0$ & $0$ & $1$
\\
$\lbrack 2 \rbrack_q {\tilde G}_1$ & $0$& $1$ & $0$ & $-1$ & $0$ & $1$ & $0$
\\
$\lbrack 2 \rbrack^2_q  {\tilde G}_2$ & $3$& $0$ & $-2$ & $0$ & $1$ & $0$ & $0$
\\
$\lbrack 2 \rbrack^3_q {\tilde G}_3$ &$0$& $-3$ & $0$ &  $1$ & $0$ & $0$ & $0$
\\
$\lbrack 2 \rbrack^4_q {\tilde G}_4$  & $-4$& $0$ & $1$ & $0$ & $0$ & $0$ & $0$
\\
$\lbrack 2 \rbrack^5_q {\tilde G}_5$ &$0$& $1$ & $0$&  $0$ & $0$ & $0$ & $0$
\\
$\lbrack 2 \rbrack^6_q {\tilde G}_6$ & $1$ & $0$& $0$&  $0$ & $0$ & $0$ & $0$
\end{tabular}}

\bigskip

\newpage
\noindent $ {\tilde G}_7$ satisfies 
\bigskip

\centerline{
0 = 
\begin{tabular}[t]{c|cccccccc}
    &
    $\lbrack 14 \rbrack_q B_{0 \delta}$ &
    $\lbrack 13 \rbrack_q B_{1 \delta}$ &
    $\lbrack 12 \rbrack_q B_{2 \delta}$ &
    $\lbrack 11 \rbrack_q B_{3 \delta}$ & 
    $\lbrack 10 \rbrack_q B_{4 \delta}$ &
    $\lbrack 9 \rbrack_q B_{5 \delta}$ &
    $\lbrack 8 \rbrack_q B_{6 \delta}$ &
    $\lbrack 7 \rbrack_q B_{7 \delta}$  
\\
\hline
$ {\tilde G}_0$ & $0$& $0$& $0$&  $0$ & $0$ & $0$ & $0$ & $1$
\\
$\lbrack 2 \rbrack_q  {\tilde G}_1$ & $-1$& $0$& $1$ & $0$ & $-1$ & $0$ & $1$ & $0$
\\
$\lbrack 2 \rbrack^2_q  {\tilde G}_2$ & $0$&  $3$& $0$ & $-2$ & $0$ & $1$ & $0$ & $0$
\\
$\lbrack 2 \rbrack^3_q {\tilde G}_3$  &$6$& $0$& $-3$ & $0$ &  $1$ & $0$ & $0$ & $0$
\\
$\lbrack 2 \rbrack^4_q  {\tilde G}_4$  &$0$& $-4$& $0$ & $1$ & $0$ & $0$ & $0$ & $0$
\\
$\lbrack 2 \rbrack^5_q {\tilde G}_5$ &$-5$& $0$& $1$ & $0$&  $0$ & $0$ & $0$ & $0$
\\
$\lbrack 2 \rbrack^6_q {\tilde G}_6$ & $0$ & $1$ & $0$& $0$&  $0$ & $0$ & $0$ & $0$
\\
$\lbrack 2 \rbrack^7_q {\tilde G}_7$ & $1$ & $0$& $0$& $0$&  $0$ & $0$ & $0$ & $0$
\end{tabular}}

\bigskip

%\newpage
\noindent ${\tilde G}_8$ satisfies $0=$
\bigskip

\centerline{
\begin{tabular}[t]{c|ccccccccc}
    &
    $\lbrack 16 \rbrack_q B_{0 \delta}$ &
    $\lbrack 15 \rbrack_q B_{1 \delta}$ &
    $\lbrack 14 \rbrack_q B_{2 \delta}$ &
    $\lbrack 13 \rbrack_q B_{3 \delta}$ & 
    $\lbrack 12 \rbrack_q B_{4 \delta}$ &
    $\lbrack 11 \rbrack_q B_{5 \delta}$ &
    $\lbrack 10 \rbrack_q B_{6 \delta}$ &
    $\lbrack 9 \rbrack_q B_{7 \delta}$ &
    $\lbrack 8 \rbrack_q B_{8 \delta}$  
\\
\hline
$ {\tilde G}_0$ & $0$& $0$& $0$& $0$&  $0$ & $0$ & $0$ & $0$ & $1$
\\
$\lbrack 2 \rbrack_q  {\tilde G}_1$ & $0$& $-1$& $0$& $1$ & $0$ & $-1$ & $0$ & $1$ & $0$
\\
$\lbrack 2 \rbrack^2_q {\tilde G}_2$ & $-4$& $0$&  $3$& $0$ & $-2$ & $0$ & $1$ & $0$ & $0$
\\
$\lbrack 2 \rbrack^3_q  {\tilde G}_3$  &$0$& $6$& $0$& $-3$ & $0$ &  $1$ & $0$ & $0$ & $0$
\\
$\lbrack 2 \rbrack^4_q  {\tilde G}_4$  &$10$& $0$& $-4$& $0$ & $1$ & $0$ & $0$ & $0$ & $0$
\\
$\lbrack 2 \rbrack^5_q  {\tilde G}_5$ &$0$& $-5$& $0$& $1$ & $0$&  $0$ & $0$ & $0$ & $0$
\\
$\lbrack 2 \rbrack^6_q  {\tilde G}_6$ & $-6$& $0$ & $1$ & $0$& $0$&  $0$ & $0$ & $0$ & $0$
\\
$\lbrack 2 \rbrack^7_q  {\tilde G}_7$ & $0$& $1$ & $0$& $0$& $0$&  $0$ & $0$ & $0$ & $0$
\\
$\lbrack 2 \rbrack^8_q  {\tilde G}_8$ & $1$ & $0$& $0$& $0$& $0$&  $0$ & $0$ & $0$ & $0$
\end{tabular}}

\bigskip

\section{Appendix B}
For the $q$-Onsager algebra $\mathcal O_q$ we use  \eqref{eq:recWnm1}, \eqref{eq:recWnm2} 
to obtain  $\lbrace W_{-n}\rbrace_{n=0}^7$ and $\lbrace W_{n+1}\rbrace_{n=0}^7$ in terms of $\lbrace B_{n\delta+\alpha_0} \rbrace_{n=0}^7$, $\lbrace B_{n\delta+\alpha_1} \rbrace_{n=0}^7$, 
 $\lbrace \tilde G_{n} \rbrace_{n=0}^7$. Recall that $\tilde G_0= -(q-q^{-1}) \lbrack 2 \rbrack^2_q$.
 
  \bigskip

\noindent We have
\medskip

\noindent $W_0= B_{\alpha_0}=
-(q-q^{-1})^{-1}\lbrack 2 \rbrack^{-2}_q
B_{\alpha_0}\tilde G_0 $.
\bigskip

\noindent $W_{-1}$ is equal to 
$-(q-q^{-1})^{-1}\lbrack 2 \rbrack^{-3}_q$ times
\bigskip

\centerline{
\begin{tabular}[t]{c|cc}
    &
${\tilde G}_0$ &
$\lbrack 2 \rbrack_q  {\tilde G}_1$
\\
\hline
    $q^{-1}  B_{\alpha_1}$ &
    $1$ & $0$ 
\\
    $ B_{\alpha_0}$ &
    $0$ & $1$ 
\\
    $q B_{\delta+\alpha_0}$ &
$1$ &$0$
\end{tabular}}
\bigskip

\newpage
\noindent $W_{-2}$ is equal to 
$-(q-q^{-1})^{-1}\lbrack 2 \rbrack^{-4}_q$ times
\bigskip

\centerline{
\begin{tabular}[t]{c|ccc}
    &
${\tilde G}_0$ &
$\lbrack 2 \rbrack_q  {\tilde G}_1$ &
$\lbrack 2 \rbrack^2_q  {\tilde G}_2$
\\
\hline
    $q^{-2}  B_{ \delta+\alpha_1}$ & 
   $1$ & $0$ &  $0$  
\\
    $q^{-1}  B_{\alpha_1}$ &
   $0$ & $1$ & $0$ 
\\
    $ B_{\alpha_0}$ &
   $2$ & $0$ & $1$ 
   \\
    $q B_{\delta+\alpha_0}$ &
$0$ & $1$ &$0$
\\
    $q^{2} B_{2 \delta+\alpha_0}$ &
$1$ & $0$ &$0$
\end{tabular}}
\bigskip

\noindent $W_{-3}$ is equal to 
$-(q-q^{-1})^{-1}\lbrack 2 \rbrack^{-5}_q$ times
\bigskip

\centerline{
\begin{tabular}[t]{c|cccc}
    &
${\tilde G}_0$ &
$\lbrack 2 \rbrack_q  {\tilde G}_1$ &
$\lbrack 2 \rbrack^2_q  {\tilde G}_2$ &
$\lbrack 2 \rbrack^3_q  {\tilde G}_3$
\\
\hline
    $q^{-3}  B_{2 \delta+\alpha_1}$ &
 $1$ & $0$ &  $0$  &$0$
\\
$q^{-2}  B_{\delta+\alpha_1}$ & 
  $0$& $1$ & $0$ &  $0$  
\\
    $q^{-1}  B_{\alpha_1}$ &
   $3$& $0$ & $1$ & $0$ 
\\
    $ B_{\alpha_0}$ &
  $0$& $2$ & $0$ & $1$ 
\\
    $q B_{\delta+\alpha_0}$ &
  $3$ & $0$ & $1$&$0$
\\
    $q^{2} B_{2\delta+\alpha_0}$ &
 $0$ & $1$ & $0$ &$0$
\\
    $q^{3}  B_{3 \delta+\alpha_0}$ &
 $1$ & $0$ & $0$ &$0$
\end{tabular}}
\bigskip

\noindent $W_{-4}$ is equal to 
$-(q-q^{-1})^{-1}\lbrack 2 \rbrack^{-6}_q$ times
\bigskip

\centerline{
\begin{tabular}[t]{c|ccccc}
    &
${\tilde G}_0$ &
$\lbrack 2 \rbrack_q  {\tilde G}_1$ &
$\lbrack 2 \rbrack^2_q  {\tilde G}_2$ &
$\lbrack 2 \rbrack^3_q  {\tilde G}_3$ &
$\lbrack 2 \rbrack^4_q  {\tilde G}_4$
\\
\hline
    $q^{-4}  B_{3 \delta+\alpha_1}$ &
 $1$& $0$& $0$ & $0$&  $0$  
 \\
    $q^{-3}  B_{2 \delta+\alpha_1}$ &
 $0$& $1$& $0$ & $0$ & $0$  
\\
    $q^{-2}  B_{\delta+\alpha_1}$ & 
 $4$& $0$& $1$ & $0$ &  $0$  
\\
    $q^{-1}  B_{\alpha_1}$ &
 $0$&  $3$& $0$ & $1$ & $0$ 
\\
    $ B_{\alpha_0}$ &
 $6$& $0$& $2$ & $0$ & $1$ 
 \\
    $q B_{\delta+\alpha_0}$ &
 $0$&  $3$ & $0$ & $1$ &$0$
\\
    $q^{2} B_{2\delta+\alpha_0}$ &
$4$ & $0$ & $1$ & $0$ &$0$
\\
    $q^{3}  B_{3 \delta+\alpha_0}$ &
 $0$ & $1$ & $0$ & $0$&$0$
\\
    $q^{4}  B_{4 \delta+\alpha_0}$ & 
$1$ & $0$ &  $0$ & $0$ &$0$
\end{tabular}}
\bigskip

\noindent $W_{-5}$ is equal to 
$-(q-q^{-1})^{-1}\lbrack 2 \rbrack^{-7}_q$ times
\bigskip

\centerline{
\begin{tabular}[t]{c|cccccc}
    &
${\tilde G}_0$ &
$\lbrack 2 \rbrack_q  {\tilde G}_1$ &
$\lbrack 2 \rbrack^2_q  {\tilde G}_2$ &
$\lbrack 2 \rbrack^3_q  {\tilde G}_3$ &
$\lbrack 2 \rbrack^4_q  {\tilde G}_4$ &
$\lbrack 2 \rbrack^5_q  {\tilde G}_5$
\\
\hline
    $q^{-5}  B_{4\delta+\alpha_1}$ &
 $1$ & $0$ & $0$& $0$&  $0$ & $0$
 \\
    $q^{-4}  B_{3 \delta+\alpha_1}$ &
 $0$& $1$& $0$& $0$ & $0$&  $0$  
\\
    $q^{-3}  B_{2 \delta+\alpha_1}$ &
$5$& $0$& $1$& $0$ & $0$ & $0$  
\\
 $q^{-2}  B_{ \delta+\alpha_1}$ & 
 $0$& $4$& $0$& $1$ & $0$ &  $0$  
\\
$q^{-1}  B_{\alpha_1}$ &
 $10$& $0$&  $3$& $0$ & $1$ & $0$ 
\\
    $ B_{\alpha_0}$ &
$0$& $6$& $0$& $2$ & $0$ & $1$ 
\\
    $q B_{\delta+\alpha_0}$ &
  $10$& $0$&  $3$ & $0$ & $1$ &$0$
\\
    $q^{2} B_{2\delta+\alpha_0}$ &
  $0$& $4$ & $0$ & $1$ & $0$ &$0$
\\
    $q^{3}  B_{3 \delta+\alpha_0}$ &
  $5$& $0$ & $1$ & $0$ & $0$&$0$
\\
    $q^{4}  B_{4 \delta+\alpha_0}$ & 
 $0$& $1$ & $0$ &  $0$ & $0$ &$0$
\\
    $q^{5}  B_{5 \delta+\alpha_0}$ &
 $1$& $0$ & $0$ & $0$ & $0$ &$0$
\end{tabular}}
\bigskip

\noindent $W_{-6}$ is equal to $-(q-q^{-1})^{-1}\lbrack 2 \rbrack^{-8}_q$ times
\bigskip

\centerline{
\begin{tabular}[t]{c|ccccccc}
    &
${\tilde G}_0$ &
$\lbrack 2 \rbrack_q  {\tilde G}_1$ &
$\lbrack 2 \rbrack^2_q  {\tilde G}_2$ &
$\lbrack 2 \rbrack^3_q  {\tilde G}_3$ &
$\lbrack 2 \rbrack^4_q  {\tilde G}_4$ &
$\lbrack 2 \rbrack^5_q  {\tilde G}_5$ &
$\lbrack 2 \rbrack^6_q  {\tilde G}_6$ 
\\
\hline
  $q^{-6}  B_{5 \delta+\alpha_1}$ &
$1$&$0$& $0$ & $0$ & $0$& $0$&  $0$ 
\\
    $q^{-5}  B_{4 \delta+\alpha_1}$ &
 $0$& $1$ & $0$ & $0$& $0$&  $0$ & $0$
\\
    $q^{-4}  B_{3 \delta+\alpha_1}$ &
$6$& $0$& $1$& $0$& $0$ & $0$&  $0$  
\\
    $q^{-3}  B_{2\delta+\alpha_1}$ &
$0$&$5$& $0$& $1$& $0$ & $0$ & $0$  
\\
    $q^{-2}  B_{\delta+\alpha_1}$ & 
$15$ & $0$& $4$& $0$& $1$ & $0$ &  $0$  
\\
    $q^{-1}  B_{\alpha_1}$ &
$0$ & $10$& $0$&  $3$& $0$ & $1$ & $0$ 
\\
    $ B_{\alpha_0}$ &
$20$ &$0$& $6$& $0$& $2$ & $0$ & $1$ 
\\
    $q B_{\delta+\alpha_0}$ &
  $0$& $10$& $0$&  $3$ & $0$ & $1$ & $0$
\\
    $q^{2} B_{2\delta+\alpha_0}$ &
 $15$& $0$& $4$ & $0$ & $1$ & $0$ &$0$
\\
    $q^{3}  B_{3 \delta+\alpha_0}$ &
 $0$&  $5$& $0$ & $1$ & $0$ & $0$ &$0$
\\
    $q^{4}  B_{4 \delta+\alpha_0}$ & 
$6$& $0$& $1$ & $0$ &  $0$ & $0$ &$0$
\\
    $q^{5}  B_{5 \delta+\alpha_0}$ &
$0$& $1$& $0$ & $0$ & $0$ & $0$ &$0$
\\
    $q^{6}  B_{6 \delta+\alpha_0}$ &
 $1$& $0$& $0$ & $0$&  $0$ & $0$ &$0$
\end{tabular}}
\bigskip

\noindent $W_{-7}$ is equal to $-(q-q^{-1})^{-1}\lbrack 2 \rbrack^{-9}_q$ times
\bigskip

\centerline{
\begin{tabular}[t]{c|cccccccc}
    &
${\tilde G}_0$ &
$\lbrack 2 \rbrack_q  {\tilde G}_1$ &
$\lbrack 2 \rbrack^2_q  {\tilde G}_2$ &
$\lbrack 2 \rbrack^3_q  {\tilde G}_3$ &
$\lbrack 2 \rbrack^4_q  {\tilde G}_4$ &
$\lbrack 2 \rbrack^5_q  {\tilde G}_5$ &
$\lbrack 2 \rbrack^6_q  {\tilde G}_6$ &
$\lbrack 2 \rbrack^7_q  {\tilde G}_7$ 
\\
\hline
    $q^{-7}  B_{6 \delta+\alpha_1}$ &
$1$&$0$& $0$ & $0$ & $0$& $0$&  $0$ & $0$
\\
$q^{-6}  B_{5 \delta+\alpha_1}$ &
$0$& $1$& $0$ & $0$ & $0$& $0$&  $0$ & $0$
\\
    $q^{-5}  B_{4 \delta+\alpha_1}$ &
$7$ &$0$& $1$& $0$& $0$ & $0$&  $0$ & $0$ 
\\
    $q^{-4}  B_{3\delta+\alpha_1}$ &
$0$ &$6$& $0$& $1$& $0$ & $0$ & $0$ & $0$ 
\\
    $q^{-3}  B_{2\delta+\alpha_1}$ & 
$21$ &$0$& $5$& $0$& $1$ & $0$ &  $0$ & $0$ 
\\
    $q^{-2}  B_{ \delta+\alpha_1}$ &
$0$& $15$& $0$&  $4$& $0$ & $1$ & $0$ & $0$
\\
$q^{-1} B_{\alpha_1}$ &
$35$& $0$& $10$& $0$& $3$ & $0$ & $1$ & $0$ 
\\
    $ B_{\alpha_0}$ 
& $0$& $20$& $0$& $6$& $0$&  $2$ & $0$ & $1$ 
\\
    $q B_{\delta+\alpha_0}$ &
$35$&  $0$& $10$& $0$&  $3$ & $0$ & $1$ &$0$
\\
    $q^{2} B_{2\delta+\alpha_0}$ &
$0$& $15$& $0$& $4$ & $0$ & $1$ & $0$ &$0$
\\
    $q^{3}  B_{3 \delta+\alpha_0}$ &
$21$& $0$&  $5$& $0$ & $1$ & $0$ & $0$ &$0$
\\
    $q^{4}  B_{4 \delta+\alpha_0}$ & 
$0$& $6$& $0$& $1$ & $0$ &  $0$ & $0$ &$0$
\\
    $q^{5}  B_{5 \delta+\alpha_0}$ &
$7$&$0$& $1$& $0$ & $0$ & $0$ & $0$ &$0$
\\
    $q^{6}  B_{6 \delta+\alpha_0}$ &
$0$& $1$& $0$& $0$ & $0$&  $0$ & $0$ &$0$
\\
    $q^{7}  B_{7 \delta+\alpha_0}$ &
$1$& $0$& $0$& $0$ & $0$&  $0$ & $0$ &$0$
\end{tabular}}
\bigskip

\noindent $W_1= B_{\alpha_1}=
-(q-q^{-1})^{-1}\lbrack 2 \rbrack^{-2}_q
B_{\alpha_1}\tilde G_0 $.
\bigskip

\noindent $W_{2}$ is equal to 
$-(q-q^{-1})^{-1}\lbrack 2 \rbrack^{-3}_q$ times
\bigskip

\centerline{
\begin{tabular}[t]{c|cc}
    &
${\tilde G}_0$ &
$\lbrack 2 \rbrack_q  {\tilde G}_1$
\\
\hline
    $q^{-1}  B_{\delta+\alpha_1}$ &
    $1$ & $0$ 
\\
    $ B_{\alpha_1}$ &
    $0$ & $1$ 
\\
    $q B_{\alpha_0}$ &
$1$ &$0$
\end{tabular}}
\bigskip

\newpage
\noindent $W_{3}$ is equal to 
$-(q-q^{-1})^{-1}\lbrack 2 \rbrack^{-4}_q$ times
\bigskip

\centerline{
\begin{tabular}[t]{c|ccc}
    &
${\tilde G}_0$ &
$\lbrack 2 \rbrack_q  {\tilde G}_1$ &
$\lbrack 2 \rbrack^2_q  {\tilde G}_2$
\\
\hline
    $q^{-2}  B_{2 \delta+\alpha_1}$ & 
   $1$ & $0$ &  $0$  
\\
    $q^{-1}  B_{\delta+\alpha_1}$ &
   $0$ & $1$ & $0$ 
\\
    $ B_{\alpha_1}$ &
   $2$ & $0$ & $1$ 
   \\
    $q B_{\alpha_0}$ &
$0$ & $1$ &$0$
\\
    $q^{2} B_{\delta+\alpha_0}$ &
$1$ & $0$ &$0$
\end{tabular}}
\bigskip

\noindent $W_{4}$ is equal to 
$-(q-q^{-1})^{-1}\lbrack 2 \rbrack^{-5}_q$ times
\bigskip

\centerline{
\begin{tabular}[t]{c|cccc}
    &
${\tilde G}_0$ &
$\lbrack 2 \rbrack_q  {\tilde G}_1$ &
$\lbrack 2 \rbrack^2_q  {\tilde G}_2$ &
$\lbrack 2 \rbrack^3_q  {\tilde G}_3$
\\
\hline
    $q^{-3}  B_{3 \delta+\alpha_1}$ &
 $1$ & $0$ &  $0$  &$0$
\\
$q^{-2}  B_{2 \delta+\alpha_1}$ & 
  $0$& $1$ & $0$ &  $0$  
\\
    $q^{-1}  B_{\delta+\alpha_1}$ &
   $3$& $0$ & $1$ & $0$ 
\\
    $ B_{\alpha_1}$ &
  $0$& $2$ & $0$ & $1$ 
\\
    $q B_{\alpha_0}$ &
  $3$ & $0$ & $1$&$0$
\\
    $q^{2} B_{\delta+\alpha_0}$ &
 $0$ & $1$ & $0$ &$0$
\\
    $q^{3}  B_{2 \delta+\alpha_0}$ &
 $1$ & $0$ & $0$ &$0$
\end{tabular}}
\bigskip

\noindent $W_{5}$ is equal to 
$-(q-q^{-1})^{-1}\lbrack 2 \rbrack^{-6}_q$ times
\bigskip

\centerline{
\begin{tabular}[t]{c|ccccc}
    &
${\tilde G}_0$ &
$\lbrack 2 \rbrack_q  {\tilde G}_1$ &
$\lbrack 2 \rbrack^2_q  {\tilde G}_2$ &
$\lbrack 2 \rbrack^3_q  {\tilde G}_3$ &
$\lbrack 2 \rbrack^4_q  {\tilde G}_4$
\\
\hline
    $q^{-4}  B_{4 \delta+\alpha_1}$ &
 $1$& $0$& $0$ & $0$&  $0$  
 \\
    $q^{-3}  B_{3 \delta+\alpha_1}$ &
 $0$& $1$& $0$ & $0$ & $0$  
\\
    $q^{-2}  B_{2 \delta+\alpha_1}$ & 
 $4$& $0$& $1$ & $0$ &  $0$  
\\
    $q^{-1}  B_{\delta+\alpha_1}$ &
 $0$&  $3$& $0$ & $1$ & $0$ 
\\
    $ B_{\alpha_1}$ &
 $6$& $0$& $2$ & $0$ & $1$ 
 \\
    $q B_{\alpha_0}$ &
 $0$&  $3$ & $0$ & $1$ &$0$
\\
    $q^{2} B_{\delta+\alpha_0}$ &
$4$ & $0$ & $1$ & $0$ &$0$
\\
    $q^{3}  B_{2 \delta+\alpha_0}$ &
 $0$ & $1$ & $0$ & $0$&$0$
\\
    $q^{4}  B_{3 \delta+\alpha_0}$ & 
$1$ & $0$ &  $0$ & $0$ &$0$
\end{tabular}}
\bigskip

\noindent $W_{6}$ is equal to 
$-(q-q^{-1})^{-1}\lbrack 2 \rbrack^{-7}_q$ times
\bigskip

\centerline{
\begin{tabular}[t]{c|cccccc}
    &
${\tilde G}_0$ &
$\lbrack 2 \rbrack_q  {\tilde G}_1$ &
$\lbrack 2 \rbrack^2_q  {\tilde G}_2$ &
$\lbrack 2 \rbrack^3_q  {\tilde G}_3$ &
$\lbrack 2 \rbrack^4_q  {\tilde G}_4$ &
$\lbrack 2 \rbrack^5_q  {\tilde G}_5$
\\
\hline
    $q^{-5}  B_{5 \delta+\alpha_1}$ &
 $1$ & $0$ & $0$& $0$&  $0$ & $0$
 \\
    $q^{-4}  B_{4 \delta+\alpha_1}$ &
 $0$& $1$& $0$& $0$ & $0$&  $0$  
\\
    $q^{-3}  B_{3 \delta+\alpha_1}$ &
$5$& $0$& $1$& $0$ & $0$ & $0$  
\\
 $q^{-2}  B_{2 \delta+\alpha_1}$ & 
 $0$& $4$& $0$& $1$ & $0$ &  $0$  
\\
$q^{-1}  B_{\delta+\alpha_1}$ &
 $10$& $0$&  $3$& $0$ & $1$ & $0$ 
\\
    $ B_{\alpha_1}$ &
$0$& $6$& $0$& $2$ & $0$ & $1$ 
\\
    $q B_{\alpha_0}$ &
  $10$& $0$&  $3$ & $0$ & $1$ &$0$
\\
    $q^{2} B_{\delta+\alpha_0}$ &
  $0$& $4$ & $0$ & $1$ & $0$ &$0$
\\
    $q^{3}  B_{2 \delta+\alpha_0}$ &
  $5$& $0$ & $1$ & $0$ & $0$&$0$
\\
    $q^{4}  B_{3 \delta+\alpha_0}$ & 
 $0$& $1$ & $0$ &  $0$ & $0$ &$0$
\\
    $q^{5}  B_{4 \delta+\alpha_0}$ &
 $1$& $0$ & $0$ & $0$ & $0$ &$0$
\end{tabular}}
\bigskip

\noindent $W_{7}$ is equal to $-(q-q^{-1})^{-1}\lbrack 2 \rbrack^{-8}_q$ times
\bigskip

\centerline{
\begin{tabular}[t]{c|ccccccc}
    &
${\tilde G}_0$ &
$\lbrack 2 \rbrack_q  {\tilde G}_1$ &
$\lbrack 2 \rbrack^2_q  {\tilde G}_2$ &
$\lbrack 2 \rbrack^3_q  {\tilde G}_3$ &
$\lbrack 2 \rbrack^4_q  {\tilde G}_4$ &
$\lbrack 2 \rbrack^5_q  {\tilde G}_5$ &
$\lbrack 2 \rbrack^6_q  {\tilde G}_6$ 
\\
\hline
  $q^{-6}  B_{6 \delta+\alpha_1}$ &
$1$&$0$& $0$ & $0$ & $0$& $0$&  $0$ 
\\
    $q^{-5}  B_{5 \delta+\alpha_1}$ &
 $0$& $1$ & $0$ & $0$& $0$&  $0$ & $0$
\\
    $q^{-4}  B_{4 \delta+\alpha_1}$ &
$6$& $0$& $1$& $0$& $0$ & $0$&  $0$  
\\
    $q^{-3}  B_{3 \delta+\alpha_1}$ &
$0$&$5$& $0$& $1$& $0$ & $0$ & $0$  
\\
    $q^{-2}  B_{2 \delta+\alpha_1}$ & 
$15$ & $0$& $4$& $0$& $1$ & $0$ &  $0$  
\\
    $q^{-1}  B_{\delta+\alpha_1}$ &
$0$ & $10$& $0$&  $3$& $0$ & $1$ & $0$ 
\\
    $ B_{\alpha_1}$ &
$20$ &$0$& $6$& $0$& $2$ & $0$ & $1$ 
\\
    $q B_{\alpha_0}$ &
  $0$& $10$& $0$&  $3$ & $0$ & $1$ & $0$
\\
    $q^{2} B_{\delta+\alpha_0}$ &
 $15$& $0$& $4$ & $0$ & $1$ & $0$ &$0$
\\
    $q^{3}  B_{2 \delta+\alpha_0}$ &
 $0$&  $5$& $0$ & $1$ & $0$ & $0$ &$0$
\\
    $q^{4}  B_{3 \delta+\alpha_0}$ & 
$6$& $0$& $1$ & $0$ &  $0$ & $0$ &$0$
\\
    $q^{5}  B_{4 \delta+\alpha_0}$ &
$0$& $1$& $0$ & $0$ & $0$ & $0$ &$0$
\\
    $q^{6}  B_{5 \delta+\alpha_0}$ &
 $1$& $0$& $0$ & $0$&  $0$ & $0$ &$0$
\end{tabular}}
\bigskip

\noindent $W_{8}$ is equal to $-(q-q^{-1})^{-1}\lbrack 2 \rbrack^{-9}_q$ times
\bigskip

\centerline{
\begin{tabular}[t]{c|cccccccc}
    &
${\tilde G}_0$ &
$\lbrack 2 \rbrack_q  {\tilde G}_1$ &
$\lbrack 2 \rbrack^2_q  {\tilde G}_2$ &
$\lbrack 2 \rbrack^3_q  {\tilde G}_3$ &
$\lbrack 2 \rbrack^4_q  {\tilde G}_4$ &
$\lbrack 2 \rbrack^5_q  {\tilde G}_5$ &
$\lbrack 2 \rbrack^6_q  {\tilde G}_6$ &
$\lbrack 2 \rbrack^7_q  {\tilde G}_7$ 
\\
\hline
    $q^{-7}  B_{7 \delta+\alpha_1}$ &
$1$&$0$& $0$ & $0$ & $0$& $0$&  $0$ & $0$
\\
$q^{-6}  B_{6 \delta+\alpha_1}$ &
$0$& $1$& $0$ & $0$ & $0$& $0$&  $0$ & $0$
\\
    $q^{-5}  B_{5 \delta+\alpha_1}$ &
$7$ &$0$& $1$& $0$& $0$ & $0$&  $0$ & $0$ 
\\
    $q^{-4}  B_{4 \delta+\alpha_1}$ &
$0$ &$6$& $0$& $1$& $0$ & $0$ & $0$ & $0$ 
\\
    $q^{-3}  B_{3 \delta+\alpha_1}$ & 
$21$ &$0$& $5$& $0$& $1$ & $0$ &  $0$ & $0$ 
\\
    $q^{-2}  B_{2 \delta+\alpha_1}$ &
$0$& $15$& $0$&  $4$& $0$ & $1$ & $0$ & $0$
\\
$q^{-1} B_{\delta+\alpha_1}$ &
$35$& $0$& $10$& $0$& $3$ & $0$ & $1$ & $0$ 
\\
    $ B_{\alpha_1}$ 
& $0$& $20$& $0$& $6$& $0$&  $2$ & $0$ & $1$ 
\\
    $q B_{\alpha_0}$ &
$35$&  $0$& $10$& $0$&  $3$ & $0$ & $1$ &$0$
\\
    $q^{2} B_{\delta+\alpha_0}$ &
$0$& $15$& $0$& $4$ & $0$ & $1$ & $0$ &$0$
\\
    $q^{3}  B_{2 \delta+\alpha_0}$ &
$21$& $0$&  $5$& $0$ & $1$ & $0$ & $0$ &$0$
\\
    $q^{4}  B_{3 \delta+\alpha_0}$ & 
$0$& $6$& $0$& $1$ & $0$ &  $0$ & $0$ &$0$
\\
    $q^{5}  B_{4 \delta+\alpha_0}$ &
$7$&$0$& $1$& $0$ & $0$ & $0$ & $0$ &$0$
\\
    $q^{6}  B_{5 \delta+\alpha_0}$ &
$0$& $1$& $0$& $0$ & $0$&  $0$ & $0$ &$0$
\\
    $q^{7}  B_{6 \delta+\alpha_0}$ &
$1$& $0$& $0$& $0$ & $0$&  $0$ & $0$ &$0$
\end{tabular}}

\bigskip

\noindent Paul Terwilliger \hfil\break
\noindent Department of Mathematics \hfil\break
\noindent University of Wisconsin \hfil\break
\noindent 480 Lincoln Drive \hfil\break
\noindent Madison, WI 53706-1388 USA \hfil\break
\noindent email: {\tt terwilli@math.wisc.edu }\hfil\break

\end{document}